\documentclass[10.5pt,reqno]{amsart}
\usepackage{longtable}
\usepackage{hyperref}
\usepackage[T1]{fontenc}
\usepackage[utf8]{inputenc}
\usepackage[english]{babel}
\usepackage{textcomp}
\usepackage{dsfont}
\usepackage{latexsym}
\usepackage{amssymb}
\usepackage{amsthm}
\usepackage{amsmath}
\DeclareMathAlphabet{\mathpzc}{OT1}{pzc}{m}{en}
\usepackage{yfonts}
\usepackage{xfrac}
\usepackage{newlfont}
\usepackage{graphicx}
\usepackage{mathtools}
\usepackage{comment}
\usepackage{indentfirst}
\usepackage{braket}
\usepackage{mathrsfs}
\usepackage{xcolor}
\usepackage{fixmath}

\usepackage{scalerel}[2014/03/10]
\usepackage[usestackEOL]{stackengine}
\newcommand{\dashint}{\,\ThisStyle{\ensurestackMath{%
			\stackinset{c}{.2\LMpt}{c}{.5\LMpt}{\SavedStyle-}{\SavedStyle\phantom{\int}}}%
		\setbox0=\hbox{$\SavedStyle\int\,$}\kern-\wd0}\int}

\DeclareMathOperator{\card}{Card}
\DeclareMathOperator{\supp}{Supp}

\DeclareMathOperator{\Hol}{Hol}

\renewcommand{\Re}{\mathrm{Re}\,}
\renewcommand{\Im}{\mathrm{Im}\,}
\newcommand{\Supp}[1]{\supp\left( #1\right) }
\newcommand{\Pfaff}{\mathrm{Pf}}
\newcommand{\ee}{\mathrm{e}}

\newcommand{\dd}{\mathrm{d}}

\DeclarePairedDelimiter{\abs}{\lvert}{\rvert}

\DeclarePairedDelimiter{\norm}{\lVert}{\rVert}

\let\originalleft\left
\let\originalright\right
\renewcommand{\left}{\mathopen{}\mathclose\bgroup\originalleft}
\renewcommand{\right}{\aftergroup\egroup\originalright}

\newcommand{\N}{\mathds{N}}
\newcommand{\Z}{\mathds{Z}}

\newcommand{\C}{\mathds{C}}

\newcommand{\R}{\mathds{R}}

\newcommand{\Ff}{\mathfrak{F}}

\newcommand{\Uf}{\mathfrak{U}}

\newcommand{\Bc}{\mathcal{B}}

\newcommand{\Ec}{\mathcal{E}}
\newcommand{\Fc}{\mathcal{F}}

\newcommand{\Hc}{\mathcal{H}}

\newcommand{\Lc}{\mathcal{L}}
\newcommand{\cM}{\mathcal{M}}  
\newcommand{\Nc}{\mathcal{N}}
\newcommand{\Oc}{\mathcal{O}}

\newcommand{\Sc}{\mathcal{S}}

\newcommand{\Ms}{\mathscr{M}}

\newcommand{\meg}{\leqslant}
\newcommand{\Meg}{\geqslant}
\newcommand{\eps}{\varepsilon}
\newcommand{\mi}{\mu}

\title[Carleson measures on Bernstein spaces]{Carleson and Sampling measures on Bernstein spaces on Siegel CR Manifolds}

\date{}
\begin{document}

\theoremstyle{definition}
\newtheorem{deff}{Definition}[section]

\newtheorem{oss}[deff]{Remark}
\newtheorem{ass}[deff]{Assumptions}
\newtheorem{nott}[deff]{Notation}
\newtheorem{exam}[deff]{Examples}

\theoremstyle{plain}
\newtheorem{teo}[deff]{Theorem}
\newtheorem{lem}[deff]{Lemma}
\newtheorem{prop}[deff]{Proposition}
\newtheorem{cor}[deff]{Corollary}
\author[M. Calzi, M. M. Peloso]{Mattia Calzi, 
Marco M. Peloso}

\address{Dipartimento di Matematica, Universit\`a degli Studi di
  Milano, Via C. Saldini 50, 20133 Milano, Italy}
\email{{\tt mattia.calzi@unimi.it}}
\email{{\tt marco.peloso@unimi.it}}

\thanks{{\em Math Subject Classification 2020} 32A15, 32A37, 32A50, 46E22.}
\thanks{The authors are members of the
  Gruppo Nazionale per l'Analisi Matematica, la Probabilit\`a e le  loro Applicazioni (GNAMPA) of the Istituto Nazionale di Alta Matematica (INdAM) }
\thanks{The authors are partially supported by the 2020  INdAM--GNAMPA grant  {\em Fractional Laplacians and subLaplacians on Lie groups and trees}.}
 \begin{abstract}
In this paper we introduce and study Carleson and sampling measures on Bernstein spaces on a class of quadratic CR manifold called {\em Siegel CR manifolds}.
These are spaces of entire functions of exponential type whose restrictions to the given Siegel CR manifold are $L^p$-integrable with respect to a natural measure.  For these spaces, we prove necessary and sufficients conditions for a Radon
measure to be a Carleson or a sampling measure.  We also provide sufficient conditions for sampling sequences.
\end{abstract}
\maketitle

\section{Introduction} 

For $\kappa>0$, let $\Ec_\kappa(\C)$ be the space of entire function of
exponential type  at most $\kappa$, that is,
\[
\Ec_\kappa(\C)=\Set{ f\in\Hol(\C)\colon \limsup_{w\to \infty} \frac{\log \abs{f(w)}}{1+\abs{w}}\meg \kappa}  .
\]
For $p\in (0,\infty]$, the classical Bernstein spaces are defined as the spaces of functions in $\Ec_\kappa(\C)$ whose restriction to the real line are in $L^p(\R)$, that is, writing $w=u+iv$ and letting $f_v\colon u\mapsto f(u+iv)$, 
\[
\Bc_\kappa^p =\Set{ f\in\Ec_\kappa(\C)\colon f_0 \in L^p(\R)},
\]
endowed with the norm $\norm{f}_{\Bc_\kappa^p}\coloneqq\norm{f_0}_{L^p}$.
If $f\in\Bc_\kappa^p$, then the Phragm\'en--Lindel\"of principle easily implies that in fact $\abs{f(w)}\le Ce^{\kappa\abs{w}}$ for all $w\in\C$.  The case $p=2$ corresponds to the classical Paley--Wiener space $PW_\kappa$, and in fact also the Bernstein spaces are sometimes reffered to as the Paley--Wiener spaces.  Cf.~\cite{PlancherelPolya,Levin-Lectures,Young} for more information on classical Bernstein spaces.

In~\cite{Bernstein}, an analogue of Bernstein spaces in several
variables is considered, where the role of $\R$ is played by a `Siegel
CR submanifold'.
More precisely, given a complex hilbertian space $E$ of dimension $n$, a real
hilbertian space $F$ of dimension $m$, and a hermitian map $\Phi\colon
E\times E\to F_\C$, we consider the quadratic (or quadric) CR manifold (cf.~\cite{Boggess,PR})
\[
\cM\coloneqq \Set{(\zeta,x+i\Phi(\zeta))\colon \zeta\in E, x\in F}= \Set{(\zeta,z)\in E\times F_\C\colon \rho(\zeta,z)=0},
\]
where 
\[
\Phi(\zeta)\coloneqq \Phi(\zeta,\zeta) \qquad \text{and} \qquad \rho(\zeta,z)\coloneqq \Im z-\Phi(\zeta)
\]
for every $(\zeta,z)\in E\times F_\C$. Then, $\cM$  is a CR manifold of CR-dimension $n$ and real codimension $m$.
The manifold $\cM$ can be canonically identified with a $2$-step nilpotent Lie group $\Nc\coloneqq E\times F$, endowed with the  product
\[
(\zeta,x)(\zeta',x')\coloneqq (\zeta+\zeta', x+x'+2 \Im \Phi(\zeta,\zeta'))
\]
for every $(\zeta,x),(\zeta',x')\in E\times F$.
Then, $\Nc$ acts freely and affinely on the complex space $E\times F_\C$ as 
\[
(\zeta,x)\cdot (\zeta',z')\coloneqq (\zeta+\zeta',x+i\Phi(\zeta)+z'+2 i\Phi(\zeta',\zeta))
\]
In particular, $\Nc$ acts simply transitively on the CR submanifold $\cM=\Nc\cdot (0,0)$. 

Then, given a compact convex subset $K$ of $F'$,  we consider the Bernstein spaces 
\[
\Bc^p_K(\Nc)\coloneqq \Set{f\in \Hol(E\times F_\C)\colon \forall h\in F\:\: \norm{f_h}_{L^\infty(\Nc)}\meg \ee^{H_K(\rho(\zeta,z))}, f_0\in L^p(\Nc)},
\]
endowed with the norm $\norm{f}_{\Bc^p_K(\Nc)}\coloneqq \norm{f_0}_{L^p(\Nc)}$, where 
\[
f_h\colon \Nc\ni (\zeta,x)\mapsto f(\zeta,x+i\Phi(\zeta)+i h)
\]
for every function $f$ on $E\times F_\C$ and for every $h\in F$, while
\[
H_K\colon F\ni h\mapsto \sup_{\lambda\in -K} \langle \lambda,h\rangle \in [-\infty,\infty)
\]
is the support function associated with $K$ (cf.~\cite[Section 4.3]{Hormander} and~\cite[Exercise 9 of \S\ 2]{BourbakiTVS}). 
Cf.~\cite{Bernstein} for other equivalent definitions and other basic properties of $\Bc^p_K(\Nc)$.

As a consequence of~\cite[Theorem 1.10]{Bernstein}, $\Bc^p_K(\Nc)=\Bc^p_{K\cap \overline{\Lambda_+}}(\Nc)$ when $p<\infty$, where
\[
\Lambda_+\coloneqq \Set{\lambda\in F'\colon \forall \zeta\in
E\setminus \Set{0}\:\: \langle \lambda , \Phi(\zeta) \rangle>0}.
\]
It is therefore natural to restrict our attention to the case in which
the open convex cone $\Lambda_+$ is non-empty, in which case $\cM$ is
said to be a `Siegel' CR submanifold of $E\times F_\C$, while $\Nc$,
endowed with the CR structure induced by $\cM$, is said to be a
`Siegel' CR manifold.

We point out that these spaces constitute a natural, highly nontrivial
multidimensional extension of the classical Bernstein spaces. In Section
\ref{sec:2} we discuss their main properties and features,
connections with the classical spaces in one-variable, with other
extenstions in several variables, and present some examples.

\medskip

We propose to investigate Carleson and sampling measures for the spaces $\Bc^p_K(\Nc)$. 
Given a Hausdorff space $X$, and quasi-Banach space $Y$ of functions
on $X$, and $p\in (0,\infty)$,  a $p$-Carleson measure for $Y$ is a
positive Radon measure $\mi$ on $X$ such that $Y\subseteq L^p(\mi)$
continuously. If, in addition, the canonical mapping $Y\to L^p(\mi)$
is an isomorphism onto its image, that $\mi$ is said to be
$p$-sampling for $Y$.
We also recall that, classically, if $Y$ is a reproducing kernel
hilbertian space,
a locally finite sequence of distinct points $(z_j)_{j\in J}\subseteq X$
is called a $2$-sampling sequence (or, simply a sampling sequence
since 
the parameter $p=2$ is understood from the context) if the measure
$\mu:=\sum_{j\in J} c_j^{-1}\delta_{z_j}$ is a $2$-sampling measure for
$Y$. Here, $c_j=k_{z_j}(z_j)$, where $k_{z_j}$ is the reproducing
kernel of $Y$ at the point $z_j$, and
$\delta_z$ denotes the Dirac delta  at $z\in
X$.
Clearly, the notion of $2$-sampling measures for $Y$ is a
generalization of the notion of ($2$-)sampling {\em sequences} for $Y$. 

In the case of the classical Bernstein spaces $\Bc^p_\kappa$,
$p\in(0,\infty]$, 
  $p$-sampling sequences {\em on the real line} were studied by Plancherel and P\'olya
  in~\cite[Nos.\ 40, 44]{PlancherelPolya}.
They proved that 
  for every $p\in(0,\infty]$ and $\kappa'>\kappa$, there exist two constants $C_{p,\kappa,\kappa'},C'_{p,\kappa,\kappa'}>0$ such that, for every $f\in\Bc^p_\kappa$,
\[
C_{p,\kappa,\kappa'}\norm{ f}_{\Bc^p_\kappa} \le \Big( \sum_{n\in\Z}
\abs{f(n\pi/\kappa')}^p \Big)^{1/p} \le C'_{p,\kappa,\kappa'} \norm{f}_{\Bc^p_\kappa} ,
\]
(modification if $p=\infty$).
If $p\in (1,\infty)$,  then one may take $\kappa'=\kappa$,
while if $p=2$, the classical Whittaker--Kotelnikov--Shannon theorem gives that
$C_{2,\kappa,\kappa}=C'_{2,\kappa,\kappa}=\sqrt{\kappa/\pi}$.
General samplings sequences for $\Bc^2_\kappa$ have been studied by
Beurling~\cite{Beurling} for sampling sequences in $\R$, and by Seip
in~\cite[Theorem 10 of Chapter 6]{Seip} for sampling sequences in
$\C$, see also~\cite{OS2}.

\medskip

Carleson measures were introduced by L.\ Carleson in~\cite{Carleson1,Carleson2} in order to study the corona problem in the classical Hardy spaces on the unit disc. The study of these measures has flourished since then, and has been generalized to several different settings, such as weighted Bergman spaces, the Dirichlet space, Fock spaces, model spaces, Bernstein spaces, etc. 
Sampling measures arose as extensions of sampling sequences.
In the context of Bernstein spaces, we mention~\cite{Panejah1,Panejah2,Lin,Kacnelson,Logvinenko,Ortega-Cerda} for sampling measures (and the particular case of dominant sets), and~\cite{PlancherelPolya,Beurling,Landau,Flornes,Seip,OlevskiiUlanovskii,MonguzziPelosoSalvatori} for sampling sequences. Cf.~\cite{Fricainetal} and the references therein for a survey on sampling (and reverse Carleson) measures on various function spaces.
Cf.~\cite{Seip} and the references therein for more information on sampling (and interpolating) sequences for various function spaces.

\medskip

The paper is structured as follows.
In Section~\ref{sec:2} we recall some basic definitions and facts
which will be needed in the following sections. We introduce the
Bernstein spaces in our setting and discuss the known results in the
classical one-dimensional case and the extensions to several variables
present in the literature.  
 In Section~\ref{sec:3}
we shall prove our main results on Carleson measures. After providing
some general sufficient (cf.~Proposition~\ref{prop:11bis}) and
necessary (cf.~Proposition~\ref{prop:1}) conditions, we characterize
the Carleson measures which are supported in $\rho^{-1}(\overline
B_F(0,R))$ for some $R>1$ (cf.~Theorem~\ref{teo:2}). 

In Section~\ref{sec:4}, we consider only measures $\mi$ supported in $\rho^{-1}(\overline B_F(0,R))$ for some $R>1$, and we provide general necessary (cf.~Proposition~\ref{prop:12}) and sufficient (cf.~Theorem~\ref{teo:1}) conditions for $\mi$ to be $p$-sampling for $\Bc^p_K(\Nc)$, we provide a number of sufficient criteria for sampling measures (cf.~Corollaries~\ref{cor:2} and~\ref{cor:3}) and for sampling sequences (cf.~Corollaries~\ref{cor:4} and~\ref{cor:5}). We then extend to this setting the known relation between sampling sequences for the various $\Bc^p_K(\Nc)$ (cf.~\cite[Theorem 2.1]{OlevskiiUlanovskii} and Proposition~\ref{prop:2}), and we specialize to our setting the general Beurling-type necessary conditions for sampling sequences proved in very general context in~\cite{Romeroetal} (cf.~Proposition~\ref{prop:3}).

\section{Bernstein Spaces on Siegel CR Manifolds}\label{sec:2}

We shall denote by $E$ a complex hilbertian space of finite dimension $n$, by $F$ a real hilbertian space of finite dimension $m$, and by $\Phi\colon E\times E\to F_\C$ a hermitian mapping so that the open convex cone
\[
\Lambda_+\coloneqq \Set{\lambda\in F'\colon \forall \zeta \in E\setminus \Set{0}\:\: \langle \lambda, \Phi(\zeta)\rangle>0}
\]
is not empty. Then, $\Phi$ is non-degenerate and $\Lambda_+$ is the interior of the polar of $\Phi(E)$. By the polar of a subset $A$ of $F$, we mean
\[
A^\circ \coloneqq \Set{\lambda\in F'\colon \forall h\in A\:\: \langle\lambda, h\rangle \Meg -1}.
\]
We define the polar of the subsets of $F'$ (identifying $F$ with $F''$) analogously, so that $A^{\circ \circ}$ is the closed convex envelope of $A\cup\Set{0}$ (cf.~\cite[Theorem 1 of Chapter II, \S\ 6, No. 3]{BourbakiTVS}). In particular, if $A\subseteq B\subseteq A^{\circ \circ}$, then $A^\circ =B^\circ$.

We define $\rho\colon E\times F_\C\ni (\zeta,x)\mapsto \Im z-\Phi(\zeta)\in F$ and identify $\Nc\coloneqq E\times F$ with the CR submanifold $\rho^{-1}(0)$ of $E\times F_\C$ (cf.~\cite{Boggess} for more information on (quadratic or quadric) CR manifolds). If we endow $\Nc$ with the $2$-step nilpotent Lie group structure induced by the product
\[
(\zeta,x)(\zeta',x')\coloneqq (\zeta+\zeta',x+x'+2\Im \Phi(\zeta,\zeta'))
\]
for every $(\zeta,x),(\zeta',x')\in \Nc$, then the CR structure of $\Nc$ is left-invariant and generated by the left-invariant vector fields $Z_v$ which induce the Wirtinger derivative $\frac 1 2 (\partial_v-i\partial_v)$, $v\in E$. Explicitly,
\[
Z_v = \frac 1 2 (\partial_v-i\partial_v)+i\Phi(v,\,\cdot\,)\partial_F.
\]
Thus, by a CR function on $\Nc$ we shall means a function $f$ of class $C^1$ such that $\overline{Z_v} f=0$ for every $v\in E$.

We may also endow $E\times F_\C$ with a $2$-step nilpotent Lie group structure induced by the product
\[
(\zeta,z)\cdot (\zeta',z')\coloneqq(\zeta+\zeta',z+z'+2 i \Phi(\zeta',\zeta)), 
\]
so that $\rho^{-1}(0)$ becomes a subgroup of $E\times F_\C$ and the mapping $\Nc\ni (\zeta,x)\mapsto (\zeta,x+i\Phi(\zeta))\in \rho^{-1}(0)$ an isomorphism.
Given a function $f$ on $E\times F_\C$, we define
\[
f_h\colon \Nc\ni (\zeta,x)\mapsto f(\zeta,x+i\Phi(\zeta)+ i h)
\]
for every $h\in F$. Given $(\zeta,z)\in E\times F_\C$, we define $L_{(\zeta,z)}f\coloneqq f((\zeta,z)^{-1}\,\cdot\,)$. We define $L_{(\zeta,x)}g$, for $(\zeta,x)\in \Nc$ and a function $g$ on $ \Nc$, analogously.

Given a compact subset $K$ of $F'$, we define $\Oc_K(\Nc)$ as the space of CR functions $f$ of class $C^\infty$ on $\Nc$ which grow polynomially with every left- (or right-)invariant derivative, such that $\Fc_F[f(\zeta,\,\cdot\,)]$ is supported in $K$ for every $\zeta\in E$, where $\Fc_F$ denotes the Fourier transform on $F$ (cf.~\cite{PWS}).  
We shall denote by $\Hc^d$ the (suitably normalized) $d$-dimensional Hausdorff measure on the relevant metric space, for every $d\in \N$. In particular, $\Hc^{2n+m}$ and $\Hc^{2n+2m}$ are left a right Haar measures on $\Nc$ and $E\times F_\C$, respectively.

\medskip

Given a compact subset $K$ of $F'$, we define 
\[
H_K\colon F\ni h\mapsto \sup_{\lambda\in -K}\langle \lambda, h\rangle \in [-\infty,\infty),
\]
so that $H_K$ is the supporting function of the convex envelope of $K$ (cf.~\cite[Section 4.3]{Hormander} or~\cite[Exercise 9 of \S\ 2]{BourbakiTVS}). In particular, $H_K=-\infty$ if and only if $K=\emptyset$, while $H_K(h)>-\infty$ for every $h\in F$ when $K\neq \emptyset$. In addition, $H_K$ is continuous and subadditive, and may be identified with the Minkowski functional (or gauge) associated with $K^\circ$ when $0\in K$.

If $K$ is a compact convex subset of $F'$, then the mapping $f\mapsto f_0$ induces a bijection of the set of $f\in \Hol(E\times F_\C)$ such that there are $N,C>0$ such that
\[
\abs{f(\zeta,z)}\meg C(1+\abs{\zeta}+\abs{z})^N \ee^{H_K(\rho(\zeta,z))}
\]
for every $(\zeta,z)\in E\times F_\C$, onto  $\Oc_K(\Nc)$ (cf.~\cite[Theorem 3.3]{PWS}). For this reason, given a (not necessarily convex) compact subset $K$ of $F'$, we define $\Hol_K(E\times F_\C)$ as the set of $f\in \Hol(E\times F_\C)$ satisfying the above estimate and such that $f_0\in \Oc_K(\Nc)$.

Notice that $\Oc_K(\Nc)=\Oc_{K\cap \overline{\Lambda_+}}(\Nc)$ for every compact subset $K$ of $F'$, thanks to~\cite[Proposition 5.7]{PWS}, so that we may reduce to considering only $K\subseteq \overline{\Lambda_+}$.

For every $p\in (0,\infty]$ and for every compact subset $K$ of $F'$, we then define
\[
\Bc^p_K(\Nc)\coloneqq \Set{f\in \Hol_K(E\times F_\C)\colon f_0\in L^p(\Nc)},
\]
endowed with the norm $f \mapsto \norm{f_0}_{L^p(\Nc)}$. This
definition agrees with the one given in~\cite{Bernstein} when $K$ is
convex (which is the only case considered therein). As before,
$\Bc^p_K(\Nc)=\Bc^p_{K\cap \overline{\Lambda_+}}(\Nc)$, so that we may
always assume that $K\subseteq \overline{\Lambda_+}$. \medskip

We now illustrate a few examples of our setting.
\begin{exam}
  First of all, the
classical Bernstein spaces $\Bc_\kappa$ considered in the Introduction
correspond to the case $n=0$, that is $E=\Set{0}$,  $m=1$, so that 
$\Nc = \R$, and $K=[-\kappa,\kappa]$.  If $K$ is a  {\em convex} compact
subset of $\R$, then the spaces $\Bc_K(\R)$ are isomorphic to
$\Bc_\kappa$ via the multplication of a suitable
character $\ee^{iaz}$.  For this classical case, see
e.g.~\cite{Levin-Lectures,Young}.   The case of a general compact
subset $K$  was considered in~\cite{Landau}, where necessary
conditions for a sequence to be sampling were established.

For the reader's convenience, we recall that when $K$ is an interval, $K=[a,b]$ for some $a\meg b$,
then the supporting function $H_K$ is given by
\[
H_K(h)= \begin{cases}
-a h & \text{if $h\Meg 0$}\\
-b h & \text{if $h\meg 0$}
\end{cases}
\] 
for every $h\in \R$. 

The case $n=0$, $m>1$ and $K$ a compact parallelotope was studied
in~\cite{PlancherelPolya}.  In all these cases, $\Nc$ is abelian and
the Fourier transform is the classical Euclidean Fourier transform.
\end{exam}

\begin{exam}
  If $n\ge1$ and $m=1$, then $\cM$ is the {\em topological} boundary
  of the Siegel upper half-space $\Set{ (\zeta,z)\in\C^n\times\C\colon    \Im z>\abs{\zeta}^2}$. In this case, $\Nc$ is the $n$-dimensional Heisenberg group $H_n$.  The spaces
$\Bc^2_K(H_n)$, with $K=[0,\tau]$ were introduced and studied
in~\cite{MonguzziPelosoSalvatori}, and the authors established a sharp
sampling theorem for a class of sequences on $\cM$.  

In the general case, the spaces 
$\Bc^p_K(\Nc)$, when
$K$ is convex,
 were introduced and studied in~\cite{Bernstein}.  
\end{exam}

We observe explicitly that $\cM$ is totally real if and only if $n=0$, and a hypersurface if and only if $m=1$. Further, $\Nc$ is abelian if and only if $n=0$.

\section{Carleson Measures}\label{sec:3}

In this section, we study the $q$-Carleson measures for the Bernstein spaces
$\Bc_K^p(\Nc)$, that is, the Radon measures $\mu$ on $\Nc$ such that
$\Bc^p_K(\Nc)$ embeds as a closed subspace of $L^q(\mi)$.

\begin{deff}
 Define $\theta\coloneqq \frac 1 2 $ if $n>0$, and $\theta\coloneqq 1$ if $n=0$.
 
 We denote by $d_\Nc$ a left-invariant $\theta$-homogeneous distance on $\Nc$, with respect to the dilations given by $t\cdot (\zeta,x)\coloneqq (t^{1/2}\zeta,t x)$. We endow $E\times F_\C$ with the  distance
 \[
 d\colon ((\zeta,z),(\zeta',z'))\mapsto  \max(d_\Nc((\zeta,\Re z),(\zeta',\Re z')),\abs{\rho(\zeta,z)-\rho(\zeta',z')}),
 \]
 which is left-invariant and $\theta$-homogeneous with respect to the dilations given by $t\cdot (\zeta,z)\coloneqq (t^{1/2}\zeta,t z)$.
 We denote by $\cM_+(E\times F_\C)$ the space of positive Radon measures on $E\times F_\C$. 
\end{deff}

\begin{deff}
Given $\delta>0$ and $R>1$, by a $(\delta,R)$-lattice on a metric space $X$ we shall mean a family  $(x_j)$ of elements of $X$ such that the balls $B_X(x_j,\delta)$ are pairwise disjoint, while the balls $\overline B_X(x_j,R\delta)$ cover $X$.

By a restricted $(\delta,R)$-lattice on $E\times F_\C$ we shall mean a family $(\zeta_{j,k},z_{j,k})_{j\in J,k\in K}$ of elements of $E\times F_\C$ such that the balls $B((\zeta_{j,k},z_{j,k}),\delta)$ are pairwise disjoint, such that the balls $\overline B((\zeta_{j,k},z_{j,k}), R\delta)$ cover $E\times F_\C$, and such that $\rho(\zeta_{j,k},z_{j,k})$ does not depend on $j\in J$ for every $k\in K$.
\end{deff}

If we define $h_k\coloneqq \rho(\zeta_{j,k},z_{j,k})$, then the balls $B_F(h_k,\delta)$ are pairwise disjoint and the balls $\overline B_F(h_k,R\delta)$ cover $F$ by our choice of $d$.

\begin{deff}
For every $\mi\in \cM_+(E\times F_\C)$ and for every $R>0$, we define
\[
M_{R}(\mi)\colon E\times F_\C\ni (\zeta,z)\mapsto \mi(\overline B((\zeta,z),R))\in \R_+.
\]
For notational convenience, we also define $M_{K,R}(\mi)=M_R(\ee^{H_K\circ \rho}\cdot \mi)$. for every compact subset $K$ of $F'$, so that $M_R(\mi)=M_{\Set{0},R}$.

We define 
\[
L^{p,q}(E\times F_\C)\coloneqq \Set{f\colon E\times F_\C\to \C\colon \text{$f$ is measurable, } \norm{h\mapsto \norm{f_h}_{L^p(\Nc)}}_{L^q(F)}<\infty},
\]
and we define $L^{p,q}_0(E\times F_\C)$ as the closure of the set of measurable step functions in $L^{p,q}(E\times F_\C)$.
We define $\ell^{p,q}(J, K)$ and $\ell^{p,q}_0(J,K)$ analogously, for any two sets $J$ and $K$.
\end{deff}

\begin{lem}
Take a  compact subset $K$ of $F'$, $R>0$ and $\mi \in \cM_+(\Nc)$. Then, $M_{K,R}(\mi)$ is upper semi-continuous.
\end{lem}

\begin{proof}
 Observe that there is a sequence $(\varphi_j)$ of elements of $C_c(E\times F_\C)$ such that $\chi_{\overline B((0,0), R)} \meg \varphi_j\meg \chi_{B((0,0),(1+2^{-j})R)}$ for every $j\in \N$, so that $M_{K,R}(\mi)$ is the pointwise infimum of the continuous functions
 \[
 (\zeta,z)\mapsto \int_{E\times F_\C} \varphi_j((\zeta,z)^{-1}(\zeta',z')) \ee^{H_{K}(\rho(\zeta',z') )}\,\dd \mi(\zeta',z')
 \]
 as $j$ runs through $\N$. 
\end{proof}

\begin{lem}\label{lem:1}
 Fix a compact subset $K$ of $F'$ and $p,q\in (0,\infty]$. Then, for every $\delta,R'>0$ and for every $R>1$ there is a constant $C>0$ such that
 \[
 \frac 1 C \norm{M_{K,R'}(\mi)}_{L^{p,q}(E\times F_\C)}\meg \norm{M_{K,R\delta}(\mi)(\zeta_{j,k},z_{j,k})}_{\ell^{p,q}(J,K)}\meg C
 \norm{M_{K,R'}(\mi)}_{L^{p,q}(E\times F_\C)}
 \]
 for every $\mi \in \cM_+(E\times F_\C)$ and for every  restricted $(\delta,R)$-lattice $(\zeta_{j,k},z_{j,k})_{j,k\in J,K}$ on $E\times F_\C$.
 
 In addition, $M_{K,R'}(\mi)\in L^{p,q}_0(E\times F_\C)$  if and only if $M_{K,R\delta}(\mi)(\zeta_{j,k},z_{j,k})\in \ell^{p,q}_0(J,K)$.
\end{lem}

The proof is analogous to those of~\cite[Lemmas 2.9 and 2.12]{NanaSehba} and~\cite[Lemma 5.1]{CP2}.

\begin{proof}
 Since $M_{K,R''}(\mi)=M_{R''}(\ee^{H_K\circ \rho}\mi)$ for every $R''>0$, we may assume that $K=\Set{0}$.
 
 \textsc{Step I.}  Let us first prove that, for every $R''>R'$ there is a constant $C_1>0$ such that
 \[
 \norm{M_{R'}(\mi)}_{L^{p,q}(E\times F_\C)}\meg \norm{M_{R''}(\mi)}_{L^{p,q}(E\times F_\C)}\meg C_1
 \norm{M_{R'}(\mi)}_{L^{p,q}(E\times F_\C)}
 \]
 for every $\mi \in \cM_+(E\times F_\C)$. The first inequality is obvious.  Then, observe that, since $d$ is homogeneous and left-invariant, $\overline B((\zeta,z),R'')$ is compact for every $(\zeta,z)\in E\times F_\C$. Then, there are $(\zeta_1,z_1),\dots, (\zeta_k, z_k)\in E\times F_\C$ such that $\overline B((0,0),R'')\subseteq \bigcup_{j=1}^k \overline B((\zeta_j,z_j),R')$, so that, by left-invariance,
 \[
 M_{R''}(\mi)\meg \sum_{j=1}^k M_{R'}(\mi)(\,\cdot\,(\zeta_j,z_j))\in L^{p,q}(E\times F_\C)
 \]
 for every $\mi \in \cM_+(E\times F_\C)$.
 Hence, 
 \[
 \norm{M_{R''}(\mi)}_{L^{p,q}(E\times F_\C)}\meg k^{\max(1/p,1/q,1)}
 \norm{M_{R'}(\mi)}_{L^{p,q}(E\times F_\C)}
 \]
 for every $\mi \in \cM_+(E\times F_\C)$.
 
 \textsc{Step II.} Let $(\zeta_{j,k},z_{j,k})_{j\in J,k\in K}$ be a restricted $(\delta,R)$-lattice on $E\times F_\C$. Observe that
 \[
 M_{R\delta}(\mi)(\zeta_{j,k},z_{j,k})\meg  M_{(R+1)\delta}(\mi)(\zeta,z)
 \]
 for every $(\zeta,z)\in B((\zeta_{j,k},z_{j,k}),\delta)$, for every $j\in J$, and for every $k\in K$, so that 
 \[
 \norm{M_{R\delta}(\mi)(\zeta_{j,k},z_{j,k})}_{\ell^p(J)}\meg \Hc^{2n+m}(B_\Nc((0,0),\delta))^{-1/p} \norm{[M_{(R+1)\delta}(\mi)]_h}_{L^p(\Nc)}
 \]
 for every $k\in K$ and for every $h\in B_F(h_k,\delta)$, since the balls $B((\zeta_{j,k},z_{j,k}),\delta)$ are pairwise disjoint. Then,
 \[
 \norm{M_{R\delta}(\mi)(\zeta_{j,k},z_{j,k})}_{\ell^{p,q}(J,K)}\meg \Hc^{2n+m}(B_\Nc((0,0),\delta))^{-1/p} \Hc^m(B_F(0,\delta))^{-1/q} \norm{M_{(R+1)\delta}(\mi)}_{L^{p,q}(E\times F_\C)}.
 \]
 By \textsc{step I}, this proves that there is a constant $C_2>0$ such that
 \[
 \norm{M_{R\delta}(\mi)(\zeta_{j,k},z_{j,k})}_{\ell^{p,q}(J,K)}\meg C_2 \norm{M_{R'}(\mi)}_{L^{p,q}(E\times F_\C)}
 \]
 for every $\mi \in \cM_+(E\times F_\C)$.
 
 \textsc{Step III.}  Take a  restricted $(\delta,R)$-lattice $(\zeta_{j,k},z_{j,k})_{j\in J,k\in K}$   on $E\times F_\C$. 
 Let us first prove that there is a constant $C_3>0$ such that
 \[
 \norm{M_{(R+1)\delta}(\mi)(\zeta_{j,k},z_{j,k})}_{\ell^{p,q}(J,K)}\meg C_3 \norm{M_{R\delta}(\mi)(\zeta_{j,k},z_{j,k})}_{\ell^{p,q}(J,K)}
 \]
 for every $\mi\in \cM_+(E\times F_\C)$. Indeed, for every $(j,k)\in J\times K$ define 
 \[
 J_{j,k}\coloneqq \Set{(j',k')\in J\times K\colon \overline B_\Nc((\zeta_{j',k'},z_{j',k'}),R\delta)\cap \overline B((\zeta_{j,k},z_{j,k}),(R+1)\delta)\neq \emptyset },
 \]
 and observe that there is $N\in \N$, \emph{depending only on $\delta$ and $R$}, such that $\card(J_{j,k})\meg N$ for every $(j,k)\in J\times K$, and such that each $(j',k')\in J\times K$ is contained in at most $N$ of the sets $J_{j,k}$. For every $k\in K$,  define $h_k\coloneqq \rho(\zeta_{j,k},z_{j,k})$ for some/every $j\in J$, and set $K_k\coloneqq \Set{k'\in K\colon \overline B_F(h_{k'},R\delta)\cap  \overline B_F(h_k,(R+1)\delta)\neq \emptyset }$. Observe that we may assume that $\card(K_k)\meg N$ for every $k\in K$, and that each $k'\in K$ is contained in at most $N$ of the sets $K_k$.
 
 Then,
 \[
 M_{(R+1)\delta}(\mi)(\zeta_{j,k},z_{j,k})\meg \sum_{(j',k')\in J_{j,k}} M_{R\delta}(\mi)(\zeta_{j',k'},z_{j',k'})
 \]
 so that
 \[
 \norm{M_{(R+1)\delta}(\mi)(\zeta_{j,k},z_{j,k})}_{\ell^p(J)}\meg \sum_{k'\in K_k} N^{\max(1,1/p)}\norm{M_{R\delta}(\mi)(\zeta_{j,k'},z_{j,k'})}_{\ell^p(J)},
 \]
 whence 
 \[
 \norm{M_{(R+1)\delta}(\mi)(\zeta_{j,k},z_{j,k})}_{\ell^{p,q}(J,K)}\meg  N^{\max(1,1/p)+\max(1,1/q)}\norm{M_{R\delta}(\mi)(\zeta_{j,k},z_{j,k})}_{\ell^{p,q}(J,K)}
 \]
 and our assertion.
 
 Then, observe that
 \[
 M_{\delta}(\mi)(\zeta,z)\meg M_{(R+1)\delta}(\mi)(\zeta_{j,k},z_{j,k})
 \]
 for every $(\zeta,x)\in \overline  B((\zeta_{j,k},z_{j,k}),R\delta)$ and for every $(j,k)\in J\times K$, so that
 \[
 \norm{M_{\delta}(\mi)}_{L^p(\mi)}\meg \Hc^{2n+m}(B_\Nc((0,0),R\delta))^{1/p} \Hc^m(B_F(0,R\delta))^{1/q} \norm{M_{(R+1)\delta}(\mi)(\zeta_{j,k},z_{j,k})}_{\ell^{p,q}(J,K)}.
 \]
 By \textsc{step I}, this proves that there is a constant $C_4>0$ such that
 \[
  \norm{M_{R'}(\mi)}_{L^{p,q}(E\times F_\C)}\meg C_4 \norm{M_{R\delta}(\mi)(\zeta_{j,k},z_{j,k})}_{\ell^{p,q}(J,K)}
 \]
 for every $\mi \in \cM_+(E\times F_\C)$. This completes the proof of the first assertion.
 
 \textsc{Step IV.} The second assertion follows from the first one, approximating $M_{R_1}(\mi)$ with $M_{R_1}(\chi_{B((0,0),R_2)}\cdot \mi)$, for $R_2\to +\infty$, for every $R_1>0$.
\end{proof}

\begin{deff}
 For every $p\in (0,\infty]$, we define $p'\coloneqq \max(1,p)'$, so that $p'=\infty$ if $p\meg 1$, while $\frac 1 p+\frac{1}{p'}=1$ otherwise.
\end{deff}

\begin{prop}\label{prop:11bis}
Take  a compact subset $K$ of $\overline{\Lambda_+}$, $p,q\in (0,\infty]$, with $q<\infty$, and $R>0$.
Then, there is a constant $C>0$ such that
\[
\norm{f}_{L^q(\mi)}\meg C\norm{M_{q K, 1}(\mi)}_{L^{(p/q)',1}(E\times F_\C)}^{1/q}   \sup_{h\in F}\ee^{-H_{K}(h)}\norm{ (\chi_{B(\supp{\mi},R')} f)_h   }_{L^p(\Nc)}
\]
for every $\mi \in \cM_+(E\times F_\C)$. 

In particular, if $M_{q K,1}(\mi)\in L^{(p/q)',1}(E\times F_\C)$ (resp.\ $M_{q K,1}(\mi)\in L^{(p/q)',1}_0(E\times F_\C)$), then $\mi$ is a (resp.\ compact) $q$-Carleson measure for $\Bc^p_K(\Nc)$.
\end{prop}

Here, we set $B(\supp{\mi},R)\coloneqq \bigcup_{(\zeta,x)\in \supp{\mi}} B((\zeta,x),R)=(\supp{\mi}) B((0,0),R)$.

The proof is based on standard techniques, cf., e.g.,~\cite{Luecking2,NanaSehba,CP2}.

\begin{proof}
 We may assume that $K$ is convex and non-empty.
 
\textsc{Step I.} Let $(\zeta_{j,k},z_{j,k})_{j\in J,k\in K}$ be a restricted $(1/4,4)$-lattice on $E\times F_\C$, and define
\[
(S_+ f)_{j,k}\coloneqq \max_{\overline B((\zeta_{j,k},x_{j,k}),1)} \ee^{-H_{K}\circ \rho}\chi_{\supp{\mi}}\abs{f}\in \R_+
\]
for every $f\in \Bc^p_K(\Nc)$.
Then,
\[
\begin{split}
 \norm{f}_{L^q(\mi)}&\meg \norm{(S_+ f)_{j,k} M_{q K,1}(\mi)(\zeta_{j,k},z_{j,k})^{1/q}  }_{\ell^q(J\times K)} \\
 &\meg \norm{S_+ f}_{\ell^{p,\infty}(J,K)} \norm{M_{q K,1}(\mi)(\zeta_{j,k},z_{j,k})^{1/q}  }_{\ell^{s,q}(J,K)}\\
 &\meg \norm{S_+ f}_{\ell^{p,\infty}(J,K)} \norm{M_{q K,1}(\mi)(\zeta_{j,k},z_{j,k})  }_{\ell^{ s/q,1}(J)}^{1/q},
\end{split}
\]
where $s\coloneqq \frac{pq}{(p-q)_+}=q \left(\frac p q\right)'$. Thanks to Lemma~\ref{lem:1}, it will then suffice to prove that there is a constant $C_1>0$ such that
\[
\norm{S_+ f}_{\ell^{p,\infty}(J,K)}\meg C_1\sup_{h\in F}\ee^{-H_{K}(h)}\norm{ (\chi_{B(\supp{\mi},R)} f)_h   }_{L^p(\Nc)}
\]
for every $f\in \Bc^p_K(\Nc)$. 

Observe first that
\[
\frac{1}{C_2}\ee^{-H_{K}(h')}\meg \ee^{-H_{K}(h)}\meg C_2 \ee^{-H_{K}(h')}
\]
for every $h,h'\in F$ such that $\abs{h-h'}\meg 1+R$, where $C_2\coloneqq \sup_{\abs{h}\meg 1+R} \ee^{H_{K}(h)}$.
Now, observe that by the subharmonicity of $\abs{f}^{\min(1,p)}$,
\[
\abs{f(0,0)}^{\min(1,p)}\meg \dashint_{B_{E\times F_\C}((0,0),R')} \abs{f}^{\min(1,p)}\,\dd \Hc^{2n+2m}
\]
for every $R'>0$ and for every $f\in \Hol(E\times F_\C)$, 
where $B_{E\times F_\C}((0,0),R')$ denotes the Euclidean ball of centre $(0,0)$ and radius $R'$  in $E\times F_\C$. Hence, choosing $R'>0$ so that $B_{E\times F_\C}((0,0),R')\subseteq B((0,0),R)$, and using the left-invariance of $\Hc^{2n+2m}$, we infer that there is a constant $C_3>0$ such that
\[
\abs{f(\zeta,z)}^{\min(1,p)}\meg C_3 \int_{B((\zeta,z),R)} \abs{f}^{\min(1,p)}\,\dd \Hc^{2n+2m}=C_3 \int_{B_F(\rho(\zeta,z),R)} \int_{B((\zeta,\Re z),R)} \abs{f_h}^{\min(1,p)}\,\dd \Hc^{2n+m}\,\dd h
\]
for every $(\zeta,z)\in E\times F_\C$ and for every $f\in \Hol(E\times F_\C)$.
Hence, there is a constant $C_4>0$ such that
\[
(S_+ f)_{j,k}^{\min(1,p)}\meg C_4 \ee^{-H_{K}(h_k)}\int_{B_F(h_k,1+R)} \int_{B_\Nc((\zeta_{j,k},\Re z_{j,k}),1+R)} \abs{(\chi_{B(\supp{\mi},R)}f)_h}^{\min(1,p)}\,\dd \Hc^{2n+m}\,\dd h
\]
for every $(j,k)\in J\times K$. 
Using the finite intersection property of the balls $B_\Nc((\zeta_{j,k},z_{j,k}),1+R)$, we then infer that there is a constant $C_5>0$ such that 
\[
\begin{split}
 \norm{(S_+ f)_{j,k}}_{\ell^p(J)}&\meg  C_5 \ee^{-H_{K}(h_k)}\int_{B_F(h_k,1+R)} \norm{(\chi_{B(\supp{\mi},R)}f)_h}_{L^p(\Nc)}\,\dd h\\
 &\meg  C_5 C_2 \Hc^m(B_F(0,R+1)) \sup_{h\in F} \ee^{-H_{K}(h)} \norm{ \chi_{B(\supp{\mi},R)} f_h}_{L^p(\Nc)}
\end{split}
\]
for every $k\in K$, whence the first assertion.

\textsc{Step II.}  Now, assume that  $M_{q K,1}(\mi)\in L^{(p/q)',1}(\Nc)$. By means of  \textsc{step I} and~\cite[Theorem 1.7]{Bernstein}, we see that
\[
\norm{f}_{L^q(\mi)}\meg C\norm{M_{q K, 1}(\mi)}_{L^{(p/q)',1}(E\times F_\C)}^{1/q}   \sup_{h\in F}\ee^{-H_{K}(h)}\norm{ f_h   }_{L^p(\Nc)}=C \norm{M_{q K, 1}(\mi)}_{L^{(p/q)',1}(E\times F_\C)}^{1/q}  \norm{f_0}_{L^p(\Nc)}
\]
for every $f\in \Bc_K^p(\Nc)$, so that $\mi$ is a $q$-Carleson measure for $\Bc_K^p(\Nc)$.

\textsc{Step III.} Finally, assume that $M_{q K,1}(\mi)\in L^{(p/q)',1}_0(\Nc)$. Then, \textsc{step II} shows that the mappings 
\[
\iota_k\colon \Bc^p_K(\Nc)\ni f \mapsto \chi_{B((0,0),k+1)} f\in L^q(\mi)
\]
converge to the inclusion mapping $\Bc^p_K(\Nc)\to L^q(\mi)$ in $\Lc(\Bc^p_K(\Nc); L^q(\mi))$, so that it will suffice to show that the $\iota_k$ are compact. Now, observe that $\Bc^p_K(\Nc)$ embeds continuously into $\Hol(E\times F_\C)$ (cf.~\cite[Corollary 3.3]{Bernstein}), and that the mapping $\Hol(E\times F_\C)\ni f \mapsto \chi_{B((0,0),k+1)} f\in L^q(\mi)$ is clearly compact for every $k\in \N$. The assertion follows.
\end{proof}

\begin{prop}\label{prop:1}
Take  a  compact subset $K$ of $\overline{\Lambda_+}$, $p,q\in (0,\infty]$, with $q<\infty$, and  $\mi\in \cM_+(E\times F_\C)$. Assume that $\Bc^p_K(\Nc)\neq \Set{0} $. 
Then, there is $R>0$ such that, if  $\mi$ is a $q$-Carleson measure for $\Bc^p_K(\Nc)$, then the mapping 
\[
(\zeta,z)\mapsto \ee^{- q H_K(-\rho(\zeta,z))}M_{R}(\mi)(\zeta,z)
\]
is bounded.

If, in addition, $K$ has a non-empty interior and $\mi$ is a compact $q$-Carleson measure for $\Bc^p_K(\Nc)$, then the above function vanishes at the point at infinity.
\end{prop}

Notice that, by~\cite[Proposition 3.4]{Bernstein}, if $K$ is convex and $p<\infty$, then $\Bc^p_K(\Nc)\neq \Set{0}$ if and only if $K$ has a non-empty interior. In addition, $\Bc^\infty_K(\Nc)\neq \Set{0}$ if and only if $K\neq \emptyset $.

The proof is analogous to those of~\cite[Theorem 3.5]{NanaSehba} and~\cite[Proposition 5.3]{CP2}.

\begin{proof}
\textsc{Step I.} By assumption, there is a constant $C>0$ such that
\[
\norm{f_0}_{L^q(\mi)}\meg C\norm{f}_{\Bc^p_K(\Nc)}
\]
for every $f\in \Bc_K^p(\Nc)$.
Take $f\in \Bc^p_K(\Nc)$ with $\norm{f_0}_{L^p(\Nc)}=1$, and observe that, by translation invariance, we may assume that there is $R>0$ such that $f(\zeta,z)\neq 0$ for every $(\zeta,z)\in \overline B((0,0),R)$. Since $\overline B((0,0),R)$ is compact, there is a constant $C'>0$ such that
\[
\frac{1}{C'}\meg \abs{f(\zeta,z)}\meg C'
\]
for every $(\zeta,z)\in \overline B((0,0),R)$.  
Define $f^{(\zeta,z)}\coloneqq f((\zeta,z)^{-1}\,\cdot\,)$ for every $(\zeta,z)\in E\times F_\C$, so that $f^{(\zeta,z)}\in \Bc^p_K(\Nc)$, $\norm{f^{(\zeta,z)}}_{\Bc^p_K(\Nc)}=\norm{f_{-\rho(\zeta,z)}}_{L^p(\Nc)}\meg \ee^{H_{K}(-\rho(\zeta,z))}$ (cf.~\cite[Theorem 1.7]{Bernstein}), and
\[
\frac{1}{C'}\abs{f(0,0)} \meg \abs{f^{(\zeta,z)}(\zeta',z')}\meg C'\abs{f(0,0)}
\]
for every $(\zeta',z')\in B((\zeta,z),R)$.
Then,
\[
\ee^{H_{K}(-\rho(\zeta,z))}C \Meg \norm{f^{(\zeta,z)}}_{L^q(\mi)}\Meg \frac{\abs{f(0,0)}}{C'} M_{R}(\mi)(\zeta,z)^{1/q}
\]
for every $(\zeta,z)\in E\times F_\C$, whence the first assertion.

\textsc{Step II.} Assume, now, that $K$ has a non-empty interior and that $\mi$ is a compact $q$-Carleson measure for $\Bc^p_K(\Nc)$. 
Then, by~\cite[Theorem 3.2 and Proposition 3.4]{Bernstein}, there is a non-zero $f\in \Bc^{\min(1,p)}_{K'}(\Nc)\subseteq \Bc^p_K(\Nc)$, where $K'$ is a closed ball contained in $K$. We may then define $R,C'$ and $f^{(\zeta,z)}$, for every $(\zeta,z)\in E\times F_\C$, as in \textsc{step I}. 
Then, fix $h\in F$ and let $\Uf$ be an ultrafilter on $E\times F_\C$ which is finer than the filter `$(\zeta, z)\to \infty $', and observe that, by the compactness of the inclusions $\Bc^p_K(\Nc)\subseteq L^q(\mi)$ and  $\Bc^p_K(\Nc)\subseteq \Hol(E\times F_\C)$, $f^{(\zeta,z)}$ have limits $g_0$ and $g_1$ in $L^q(\mi)$ and in $\Hol(E\times F_\C)$ along $\Uf$, respectively. Since convergence in $L^q(\mi)$ implies convergence in measure, it is clear that $g_0=g_1$ $\mi$-almost everywhere. In addition, observe that $\lim_{(\zeta,z)\to \infty} \ee^{-H_K(-\rho(\zeta,z))} f^{(\zeta,z)}=0$ pointwise, by~\cite[Theorem 3.2]{Bernstein}, so that $g_1=0$. Hence, the arguments of \text{step I} imply that $\lim_{(\zeta,z), \Uf} \ee^{-q H_K(-\rho(\zeta,z))} M_{R}(\mi)(\zeta,z)=0$. By the arbitrariness of $\Uf$, this implies that $\ee^{-q H_K(-\rho(\,\cdot\,))} M_{R}(\mi)\in L^\infty_0(E\times F_\C)$. 
\end{proof}

\begin{cor}\label{cor:1}
 Take a compact subset $K$ of $\overline{\Lambda_+}$, $p,q\in (0,\infty)$ with $p\meg q$, and  $\mi\in \cM_+(E\times F_\C)$ such that $\rho(\supp{\mi})$ is bounded. Assume that $\Bc^p_K(\Nc)\neq \Set{0}$. Then, $\mi$ is a (resp.\ compact) $q$-Carleson measure for $\Bc^p_K(\Nc)$ if and only if $M_1(\mi)\in L^\infty(E\times F_\C)$ (resp.\ $M_1(\mi)\in L^\infty_0(E\times F_\C)$).
\end{cor}

\begin{proof}
 The assertion follows from Lemma~\ref{lem:1} and Propositions~\ref{prop:11bis} and~\ref{prop:1}.
\end{proof}

From now on, we shall restrict ourselves to measures supported in a set of the form $\rho^{-1}(\overline B_F(0,R))$ for some $R>0$, since there are no simple criteria  to determine general Carleson measures as in the one-dimensional case. 
More precisely, assume that $\Nc=F$ and that $K$ is polyhedral for simplicity. We may then assume that $0$ is an extreme point of $K$. Let $C$ be the corresponding tangent cone (that is, $\R_+ K$), and $\Omega$ the interior of its polar (which is the set where $H_K=0$). Observe that, if $\mi$ is the vague limit on $F+i \Omega$ of a sequence of measures of the form $r_k^{q m/p} (r_k\,\cdot\,)_*\mi_k$, with $r_k\to 0^+$ and the $\mi_k$ uniformly $q$-Carleson for $\Bc^p_K(F)$, then $\mi$ is $q$-Carleson for $H^p(F+i\Omega)$. Conversely, if $\mi$ is concentrated in $F+i  \Omega$ and $q$-Carleson for $H^p(F+i \Omega)$, then it is $q$-Carleson for $\Bc^q_K(F)$. Thus, the problem of determining Carleson measures for $\Bc^q_K(F)$ is essentially equivalent to the problem of determining Carleson measures for $H^p(F+i\Omega)$, where $\Omega$ runs through the set of polars of the tangent cones to $K$ at its extremal points.

Consequently, already when $K$ is a parallelotope, determining $p$-Carleson measures for $\Bc^p_K(F)$ would require imposing a condition of the form $\int_{T(U)}\ee^{p H_K(\Im z)}\,\dd \mi(z)\meg C \Hc^m(U)$ for every open connected subset $U$ of $F$, where $T(U)$ is a suitable `tent' on $U$, adapted to $K$. The situation for general polyhedral cones is even less clear, whereas for general cones even this loose connection with Hardy spaces is no longer applicable.

\begin{teo}\label{teo:2}
 Take a compact subset $K$ of $\overline{\Lambda_+}$ with a non-empty interior, $p,q\in (0,\infty]$ with $q<\infty$, and $\mi\in \cM_+(E\times F_\C)$ such that $\rho(\supp{\mi})$ is bounded. Then, $\mi$ is a (resp.\ compact) $q$-Carleson measure for $\Bc^p_K(\Nc)$ if and only if $M_1(\mi)\in L^{(p/q)'}(E\times F_\C)$ (resp.\ $M_1(\mi)\in L^{(p/q)'}_0(E\times F_\C)$).
\end{teo}

The proof is based on a technique developed in~\cite{Luecking1},  and then applied also in~\cite{NanaSehba,CP2}.

\begin{proof}
 One implication follows from Proposition~\ref{prop:11bis}. In addition, by Corollary~\ref{cor:1}, we may reduce to the case in which $q<p$, in which case $L^{(p/q)'}_0(E\times F_\C)=L^{(p/q)'}(E\times F_\C)$. Then, assume that $\mi$ is a $q$-Carleson measure for $\Bc^p_K(\Nc)$, and let us prove that $M_1(\mi)\in L^{(p/q)'}(E\times F_\C)$.
 
 Fix $f\in \Bc^{\min(1,p)}_K(\Nc)$ such that $\norm{f}_{\Bc^{\min(1,p)}_K(\Nc)}=1$ (cf.~\cite[Proposition 3.4]{CalziPeloso}). Observe that, up to a translation, we may assume that $f(\zeta,z)\neq 0$ for every $(\zeta,z) \in \overline B((0,0),R)$ for some $R>0$. Set $C_1\coloneqq \max_{\overline B((0,0),R)}\abs{f_0}/\min_{\overline B((0,0),R)}\abs{f_0}$.
 Take an $(R/4,2)$-lattice $(\zeta_j,x_j)_{j\in J}$ on $\Nc$, and let us prove that the mapping
 \[
 \Psi\colon \ell^p(J)\ni \lambda \mapsto \sum_j \lambda_j L_{(\zeta_j,x_j+i\Phi(\zeta_j))} f\in \Bc^p_K(\Nc)
 \]
 is well defined and continuous. If $p\meg 1$, then clearly
 \[
 \norm{\Psi(\lambda)}_{\Bc^p_K(\Nc)}\meg \norm*{ \abs{\lambda_j}\norm{L_{(\zeta_j,x_j+i\Phi(\zeta_j))} f}_{\Bc^p_K(\Nc)}  }_{\ell^p(J)}=\norm{\lambda}_{\ell^p(J)},
 \]
 whence our claim in this case.
 If, otherwise, $p>1$, then observe that, setting $c_R\coloneqq \Hc^{2n+m}(B_\Nc((0,0),R/2))$,
 \[
 \begin{split}
  \norm{\Psi(\lambda)}_{\Bc^p_K(\Nc)}& \meg \norm*{\sum_j \abs{\lambda_j} L_{(\zeta_j,x_j)} \abs{f_0}}_{L^p(\Nc)}\\
   &\meg \frac{C_1}{c_R}\norm*{\sum_j \int_\Nc \abs{\lambda_j} \chi_{B_\Nc((\zeta_j,x_j),R/2)}(\zeta,x) L_{(\zeta,x)}\abs{f_0}\,\dd (\zeta,x)   }_{L^p(\Nc)}\\
   &= \frac{C_1}{c_R}\norm*{\Big(\sum_j  \abs{\lambda_j} \chi_{B_\Nc((\zeta_j,x_j),R/2)}\Big)*\abs{f_0} }_{L^p(\Nc)}\\
   &\meg \frac{C_1}{c_R}\norm*{\sum_j  \abs{\lambda_j} \chi_{B_\Nc((\zeta_j,x_j),R/2)}}_{L^p(\Nc)}\\
   &\meg \frac{C_1}{c_R^{1/p'}}\norm*{\lambda}_{\ell^p(J)}
 \end{split}
 \]
 by Young's inequality, whence our claim also in this case.

 Now, take a probability space $(X,\nu)$ and a countable family $(r_j)_{j\in J}$ of $\nu$-measurable functions on $X$ such
 that 
 \[
 \Big(\bigotimes_{j\in J'} r_j \Big)(\nu)=\frac{1}{2^{\card(J')}}\sum_{\eps\in \Set{-1,1}^{J'}} \delta_\eps
 \]
 for every finite subset $J'$ of $J$ (cf.~\cite[C.1]{Grafakos}). By
 Khintchine's inequality, there is a constant $C_2>0$ such that 
 \[
 \frac{1}{C_2}\Big( \sum_{j\in J} \abs{a_j}^2 \Big)^{1/2}\meg\norm*{ \sum_{j\in J} a_j r_j}_{L^{q}(\nu)}\meg C_2 \Big( \sum_{j\in J} \abs{a_j}^2\Big)^{1/2} 
 \]
 for every $(a_j)\in \C^{(J)}$ (cf.~\cite[C.2]{Grafakos}). By the assumptions and the continuity of $\Psi$, there is a constant $C_3>0$ such that, for every
 $\lambda\in \C^{(J)}$, 
 \[
 \norm{\Psi((r_j \lambda_j)_j)}_{L^{q}(\mi)}\meg C_3  \norm{(r_j \lambda_j)_j}_{\ell^{p}(J)}= C_3 \norm{\lambda}_{\ell^{p}(J)} 
 \]
 $\nu$-almost everywhere.
 Therefore, by means of Tonelli's theorem we see that
 \[
 \norm*{\left( \sum_{j} \abs*{\lambda_{j}  L_{(\zeta_j,x_j+i\Phi(\zeta_j))}f}^2 \right)^{1/2}}_{L^{q}(\mi)}\meg C_2 C_3\norm{\lambda}_{\ell^{p}(J)} 
 \]
 for every $\lambda\in \C^{(J)}$. 
 Now, observe that there is $N\in\N$ such that $\sum_j \chi_{B((\zeta_j,x_j+i\Phi(\zeta_j)),R)}\meg N$ on $E\times F_\C$.
 Then,
 \[
 \begin{split}
  \norm*{\left( \sum_{j} \abs*{\lambda_{j}  L_{(\zeta_j,x_j+i\Phi(\zeta_j))}f}^2 \right)^{1/2}}_{L^{q}(\mi)}
  &\Meg \frac{\abs{f(0,0)}}{C_1}\norm*{\left( \sum_{j} \abs*{\lambda_{j}} \chi_{B((\zeta_j,x_j+i\Phi(\zeta_j)),R)} \right)^{1/2}}_{L^{q}(\mi)}\\ 
  &\Meg  \frac{\abs{f(0,0)}N^{-(1-p/2)_+}}{C_1}\left( \sum_{j} \abs*{\lambda_{j}}^q M_{R}(\mi)(\zeta_j,x_j+i\Phi(\zeta_j))\right)^{1/q}
 \end{split}
 \]
 for every $\lambda\in \C^{(J)}$. Using the natural duality between $\ell^{p/q}_0(J)$ and $\ell^{(p/q)'}(J)$, we then see that
 \[
 M_{R}(\mi)(\zeta_j,x_j+i\Phi(\zeta_j))\in \ell^{(p/q)'}(J)  .
 \]
 Therefore, by means of Lemma~\ref{lem:1}, we see that $M_1( (\chi_{B_F(0,R/2)}\circ \rho)\cdot \mi)\in L^{(p/q)'}(E\times F_\C)$. Applying the preceding arguments to (a finite number of) the translates $L_{(0,i h)}\mi$ of $\mi$ (which are still necessarily $q$-Carleson for $\Bc^p_K(\Nc)$), $h\in F$, the assertion follows. 
\end{proof}

We now illustrate a consequence of the theorem in the setting of the
Bernstein spaces $\Bc_K^p(\R)$, when $K$ is an interval in $\R$. 
It depends on the well-known characterization
for Carleson measures on Hardy spaces and it is proved along the lines of~\cite[Chapter VI,
\S\ 2]{Seip} for the case in which $p=q=2$ and $\mi$ is a (locally
finite) sum of distinct Dirac deltas.

For the sake of simplicity, we do not consider the case $q<p$, since
the extension of the corresponding assertion for the Hardy space
$H^p(\C_+)$  is more involved.  We denote by $\C_\pm=\R\pm i \R_+^*$. 

\begin{prop}\label{prop:11}
	Assume that $E=\Set{0}$ and $F=F'=\R$. Take $p,q\in (0,\infty)$, $p\meg q$, a compact interval $K$ in $\R$, and $\mi\in \cM_+(\C)$. Then, $\mi$ is $q$-Carleson for $\Bc^p_K(\R)$ if and only if there is a constant $C>0$ such that
 \[
 M_{q K, R}(\mi)\meg C R^{q/p}
 \]
 on $\R$, for every $R\Meg 1$.
\end{prop}

\begin{proof}
The condition is sufficient. Observe that, by
 Theorem~\ref{teo:2} below, we may assume that $\mi$ is
supported in $\Set{z\in \C\colon \abs{\Im z}\Meg 1}$. Then,
take $\alpha,\beta\in \R$ so that $K=[\alpha,\beta]$, and
observe that $\ee^{-i\langle \alpha,\,\cdot\,\rangle}
\Bc^p_K (\R)  \subseteq H^p(\C_+)$ while $\ee^{-i\langle
  \beta,\,\cdot\,\rangle} \Bc^p_K (\R) \subseteq H^p(\C_-) $. Since the
restrictions of $\ee^{-\langle \alpha,\,\cdot\,\rangle}\mi $ and
$\ee^{-\langle \beta,\,\cdot\,\rangle}\mi$ to $\C_+$ and $\C_-$,
respectively, are $q$-Carleson measures for $H^p(\C_+)$ and
$H^p(\C_-)$, respectively, thanks to the assumptions
and~\cite[7.E.II]{Luecking}, it is clear that $\mi$ is a $q$-Carleson
measure for $\Bc^p_K(\R)$. 

 The condition is necessary. By assumption, there is a constant $C_1>0$ such that
 \[
 \norm{f}_{L^q(\mi)}\meg C_1\norm{f_0}_{L^p(\R)}
 \]
 for every $f\in \Bc^p_K(\R)$. 
 Now, take $k\in\N$ so that $ k p>1$ and fix $\varphi\colon \R\to \R$ so that $\varphi(\lambda)=\frac{1}{(k-1)!}(\lambda-\alpha)_+^{k-1} $ for every  $\lambda\in (-\infty, (\alpha+\beta)/2]$, while $\varphi$ is of class $C^\infty$ on $(\alpha,+\infty) $ and $\varphi(\lambda)=0$ for $\lambda \Meg \beta$. Define
 \[
 f^{(z)}\colon \C\ni w \mapsto \int_\alpha^\beta (\lambda-\alpha)^k \varphi(\lambda) \ee^{i\langle \lambda, w-\overline z\rangle}\,\dd \lambda
 \]
 for every $z\in \C$. Observe that, integrating by parts,
 \[
 f^{(z)}(w)= \frac{(-1)^k \ee^{i\langle \alpha, w-\overline z\rangle}}{[i(w-\overline z)]^k}\left( 1 +\int_0^{\beta-\alpha} \varphi^{(k)}(\lambda+\alpha)\ee^{i\langle \lambda ,w-\overline z\rangle}\,\dd \lambda  \right)
 \]
 for every $w,z\in\C$ with $w\neq \overline z$. Observe that, for every $w,z\in\C$,
 \[
 \begin{split}
 \abs*{\int_0^{\beta-\alpha} \varphi^{(k)}(\lambda+\alpha)\ee^{i\langle \lambda w-\overline z\rangle}\,\dd \lambda }&\meg \norm{\varphi^{(k)}}_{L^\infty(\R)} \frac{\ee^{-\langle (\beta-\alpha)/2, \Im w+\Im z\rangle}-\ee^{-\langle \beta-\alpha, \Im w+\Im z\rangle}}{\Im w+\Im z},
 \end{split}
 \]
 which we may assume to be $\meg \frac 1 2 $ if $\Im w+\Im z\Meg M$ for some $M>0$.
 Then,
 \[
 \norm{f^{(z)}_0}_{L^p(\R)}^p\meg \frac{C_2}{(\Im z)^{k p-1}}\ee^{-p\langle\alpha,\Im z\rangle}
 \]
 for every $z\in \C$ with $\Im z\Meg M$, where $C_2\coloneqq \frac{3^p}{2^p} \int_\R \frac{1}{\abs{x+i}^{kp}}\,\dd x$. Analogously,
 \[
 \norm{f^{(z)}}_{L^q(\mi)}^q\Meg \frac{\ee^{ -q\langle \alpha, \Im z\rangle}}{2} \int_{B(\Re z,R)\cap\overline{\C_+}} \frac{\ee^{-q\langle\alpha, \Im w \rangle}}{2^{k q /2}(R+\Im z)^{k q} }\,\dd \mi(w)
 \]
 for every $z\in \C$ with $\Im z\Meg M$, and for every $R>0$. Choosing $z=x+i R$, this implies that
 \[
 \int_{B(x, R)\cap \overline{\C_+}} \ee^{-q \langle \alpha, \Im w\rangle}\,\dd \mi(w)\meg 2^{1+3 k p/2} C_2  R
 \]
 for every $x\in \R$ and for every $R\Meg M$. Thus,
 \[
 M_{q K,R}(\chi_{\overline{\C_+}}\cdot \mi)\meg 2^{3k q/2+1} C_2^{q/p} R^{q/p}
 \]
 on $\R$, for every $R\Meg M$. In a similar way, one shows that there is a constant $C_3>0$ such that
 \[
 M_{q K,R}(\chi_{\overline{\C_-}}\cdot \mi)\meg C_3 R^{q/p}
 \]
 on $\R$, for $R$ sufficiently large. By means of Proposition~\ref{prop:1} below, this completes the proof. 
\end{proof}

\section{Sampling Measures}\label{sec:4}

We keep the notation of Section~\ref{sec:3}. We shall only consider measures supported in $\rho^{-1}(B)$ for some bounded subset $B$ of $F$. Consequently, we shall no longer consider $M_{K,R}(\mi)$, but only $M_R(\mi)$.

\begin{deff}
 Take a compact subset  $K$ of $\overline{\Lambda_+}$, $p\in (0,\infty)$ and $\mi\in \cM_+(E\times F_\C)$. We say that $\mi$ is $p$-sampling for $\Bc^p_K(\Nc)$ if $\Bc^p_K(\Nc)\subseteq L^p(\mi)$ continuously and the canonical mapping $\Bc^p_K(\Nc)\to L^p(\mi)$ is an isomorphism onto its image.
 
 We say that $\mi$ is strongly $p$-sampling for $\Bc^p_K(\Nc)$ if it is $p$-sampling for $\Bc^p_K(\Nc)$ and 
 \[
 \Bc^{p}_K(\Nc)=\Set{f\in \Hol_K(E\times F_\C)\colon f\in L^p(\mi)  }.
 \]
 
 We say that a locally finite\footnote{This condition is only added to ensure that the corresponding measure be Radon.} family $(\zeta_j,z_j)$ of elements of $E\times F_\C$ is (strongly) sampling for $\Bc^p_K(\Nc)$ if the  measure $\sum_j \delta_{(\zeta_j,z_j)}$ is (strongly) $p$-sampling for $\Bc^p_K(\Nc)$.
\end{deff}

Let us briefly comment on the notion of a strongly $p$-sampling measure. On the one hand, $p$-sampling measures allow to reconstruct (up to constants) the quasi-norm of a holomorphic function \emph{which is known to belong to  the Bernstein space $\Bc^p_K(\Nc)$}. On the other hand, strongly $p$-sampling measures allow to verify if a holomorphic function $f\in \Hol_K(E\times F_\C)$ belongs to the Bernstein space $\Bc^p_K(\Nc)$. Obviously, some (growth) conditions on $f$ have to be imposed, unless the measure $\mi$ satisfies very strong properties (in particular, $\rho(\mi)$ must be unbounded, a case which is \emph{not} considered below). 

Notice that the sampling results proved in~\cite{PlancherelPolya} concern actually strongly sampling sequences, and hold uniformly for every $p\in (0,\infty]$. Sharp sampling results are then obtained by a limiting procedure for $p\in (1,\infty)$.

\begin{deff}
 Take a compact subset  $K$ of $\overline{\Lambda_+}$ and $\eps>0$. Then, we set $K_\eps\coloneqq K+(\overline B_{F'}(0,\eps)\cap \overline{\Lambda_+})$.
\end{deff}

\begin{lem}\label{lem:2}
 Take a compact subset  $K$ of $\overline{\Lambda_+}$, $p\in (0,\infty)$ and $\mi\in \cM_+(E\times F_\C)$ such that $\rho(\supp{\mi})$ is bounded. If $\mi$ is $p$-sampling for $\Bc^p_{K_\eps}(\Nc)$, then $\mi$ is strongly $p$-sampling for $\Bc^p_K(\Nc)$.
\end{lem}

\begin{proof}
 Fix $g\in \Bc^\infty_{\overline B_{F'}(0,\eps)\cap\overline{\Lambda_+}}(\Nc) $ so that $g_0\in \Sc(\Nc)$ and $g(0,0)=\norm{g_0}_{L^\infty(\Nc)}=1$ (cf.~\cite[Theorem 4.2 and Proposition 5.2]{PWS}). Take $f\in \Hol_K(E\times F_\C)$, so that $f g^{(j)}\in \Hol_{K_\eps}(E\times F_\C)$ for every $j\in \N$, where $g^{(j)}\coloneqq g(2^{-j}\,\cdot\,)$. In addition, $f_0 g^{(j)}_0\in \Sc(\Nc)$, so that $f g^{(j)}\in \Bc^p_{K_\eps}(\Nc)$. Then, there is a constant $C>0$ such that
 \[
 \norm{f_0 g^{(j)}_0}_{L^p(\Nc)}\meg C \norm{f g^{(j)}}_{L^q(\mi)}\meg C\norm{f}_{L^q(\mi)} \sup_{h\in\rho(\supp{\mi})} \ee^{\eps \abs{h}}
 \]
 for every $j\in \N$ (cf.~\cite[Theorem 1.7]{Bernstein}). Passing to the limit for $j\to \infty$, this proves that $f_0\in L^p(\Nc) $. The assertion follows.
\end{proof}

\begin{deff}
 We say that $\mi\in \cM_+(E\times F_\C)$ is sparse if there is $R$ such that for every $\eps>0$ there is $R'_\eps>0$ such that $M_{R}(\mi)\meg \eps$ on $(E\times F_\C)\setminus B((0,0),R'_\eps)$.
\end{deff}

\begin{prop}\label{prop:4}
 Take a compact subset $K$ of $\overline{\Lambda_+}$, $\eps>0$, $p\in (0,\infty)$, and  $\mi,\mi'\in \cM_+(E\times F_\C)$ such that $\rho(\Supp{\mi+\mi'})$ is bounded. If $\mi+\mi'$ is $p$-sampling for $\Bc^p_{K_\eps}(\Nc)$ and $\mi'$ is sparse, then $\mi$ is strongly $p$-sampling for $\Bc^p_K(\Nc)$.
\end{prop}

This result is based on~\cite[Lemma 4]{Lin}.

\begin{proof}
 Since $\mi+\mi'$ is $p$-sampling for $\Bc^p_{K_\eps}(\Nc)$, there is a constant $C_1>0$ such that
 \[
 \norm{f_0}_{L^p(\Nc)}^p\meg C_1 \norm{f}_{L^p(\mi+\mi')}^p= C_1(\norm{f}_{L^p(\mi)}^p+\norm{f}_{L^p(\mi')}^p )
 \]
 for every $f\in \Bc^p_{K_\eps}(\Nc)$. Choose $R>0$ so that for every $\eps>0$ there is $R'_\eps>0$ such that $M_{R}(\mi')\meg \eps$ on $(E\times F_\C)\setminus B((0,0),R'_\eps)$. 
 In addition, by Proposition~\ref{prop:11bis} there is a constant $C_2>0$ such that
 \[
 \norm{f}_{L^p(\mi'')}\meg C_2 \norm{M_R(\mi'') }_{L^\infty(\Nc)}^{1/p} \norm{f_0}_{L^p(\Nc)}
 \]
 for every $f\in \Bc^p_{K_\eps}(\Nc)$ and for every $\mi''\in \cM_+(E\times F_\C)$.
 Take $\eps>0$ so that $C_1 C_2^p \eps<\frac 1 2$, and observe that
 \[
 \begin{split}
  \norm{f}^p_{L^p(\mi')}\meg \int_{B((0,0),R)} \abs{f}^p\,\dd \mi'+  C_2^p \eps\norm{f_0}_{L^p(\Nc)}^p
 \end{split}
 \]
 for every $f\in  \Bc^p_{K_\eps}(\Nc)$, so that
 \[
 \norm{f_0}_{L^p(\Nc)}^p\meg 2C_1 \left(\int_{B((0,0),R)} \abs{f}^p\,\dd \mi'+ \norm{f}_{L^p(\mi)}^p\right) 
 \]
 for every $f\in  \Bc^p_{K_\eps}(\Nc)$. 
 Now, let us prove that the canonical mapping $\Bc^p_K(\Nc)\to L^p(\mi)$ is one-to-one. 
 Assume, by contradiction, that there is $f\in \Bc^p_K(\Nc)$ with $\norm{f_0}_{L^p(\Nc)}=1$ and $\norm{f}_{L^p(\mi)}=0$. 
 Observe that the mapping $\Bc^\infty_{\overline B_{F'}(0,\eps)\cap\overline{\Lambda_+}}(\Nc)\ni g\mapsto  f g\in \Bc^p_{K_\eps}(\Nc)$ is continuous and one-to-one by holomorphy. 
 By means of Riesz's lemma we may then find, by induction, a sequence $(g^{(k)})$ of elements of $\Bc^\infty_{\overline B_{F'}(0,\eps)\cap\overline{\Lambda_+}}(\Nc)$ such that the sequence $\norm{f_0 g^{(k)}_0}_{L^p(\Nc)}=1$ and $\norm{f_0 (g^{(k)}-g^{(k')})_0}_{L^p(\Nc)}\Meg \frac 1 2 $ for every $k\in \N$ and for every $k'<k$. 
 Since $ f g^{(k)}\in \Bc^p_{K_\eps}(\Nc)$ for every $k\in \N$ and since $\Bc^p_{K_\eps}(\Nc)$ embeds continuously in $\Hol(E\times F_\C)$ (cf.~\cite[Corollary 3.3]{Bernstein}), we may assume that $(f g^{(k)})$ converges locally uniformly on $E\times F_\C$, so that $\int_{B((0,0),R)} \abs{f g^{(k)}-f g^{(k+1)}}^p\,\dd \mi'\to 0$ for $k\to \infty$. 
 In addition, $\norm{f g^{(k)}}_{L^p(\mi)}=0$  for every $k\in \N$. Thus,
 \[
 \frac 1 2 \meg \norm{f_0 (g^{(k)}-g^{(k+1)})_0}_{L^p(\Nc)}\meg 2C_1 \left(\int_{B((0,0),R)} \abs{f(g^{(k)}-g^{(k+1)})}^p\,\dd \mi'\right) \to 0
 \]
 for $k\to \infty$: contradiction. Thus, the canonical mapping $\Bc^p_K(\Nc)\to L^p(\mi)$ is one-to-one.
 
 Now, assume by contradiction that the canonical mapping $\Bc^p_K(\Nc)\to L^p(\mi)$ is not an isomorphism onto its image. Then, there is a sequence $(f^{(j)})$ of elements of $\Bc^p_K(\Nc)$ such that $\norm{f^{(j)}_0}_{L^p(\Nc)}=1$ for every $j\in \N$, while $\norm{f^{(j)}}_{L^p(\mi)}\to 0$. As before, we may assume that $f^{(j)}\to f$ locally uniformly for some $f\in \Hol_K(E\times F_\C)$, so that $\norm{f_0}_{L^p(\Nc)}\meg 1 $, $f\in \Bc^p_K(\Nc)$, and $\norm{f}_{L^p(\mi)}=0$. Then, $f=0$. Therefore,
 \[
 1=\lim_{j\to \infty}\norm{f_0^{(j)}}_{L^p(\Nc)}^p\meg 2C_1 \lim_{j\to \infty} \left(\int_{B((0,0),R)} \abs{f^{(j)}}^p\,\dd \mi'+ \norm{f^{(j)}}_{L^p(\mi)}^p\right) =0,
 \]
 which is absurd. The proof is then completed by means of Lemma~\ref{lem:2}.
\end{proof}

We now prove a necessary condition for sampling measures, which is based on~\cite[Theorem 4.3]{Luecking3} (cf.~also~\cite[Proposition 7.2]{CP2}).

\begin{prop}\label{prop:12}
Take a compact subset $K$ of $\overline{\Lambda_+}$, $p\in (0,\infty)$, and $\mi\in \cM_+(E\times F_\C)$ such that $\rho(\supp{\mi})$ is bounded. 
Assume that  $\Bc^p_K(\Nc)\neq \Set{0}$ and that $\mi$ is a $p$-sampling measure for $\Bc^p_K(\Nc)$. Then, $M_1(\mi)$ is bounded and  there are $R,C>0$ such that 
\[
[M_{R}(\mi)]_0\Meg C
\] 
on $\Nc$.
\end{prop}

\begin{proof}
 The first assertion follows from Corollary~\ref{cor:1}.
 Take $R'>0$ such that $\rho(\supp{\mi})\subseteq \overline B_F(0,R')$.
By  Proposition~\ref{prop:11bis},  there is a constant $C_1>0$ such that
\[
\norm{f}_{L^p(\mi')}\meg C_1 \norm{M_{1}(\mi')}_{L^\infty(E\times F_\C)}^{1/p} \sup_{\abs{h}\meg R'+1}\norm{[\chi_{B(\supp{\mi'},1)}f]_h}_{L^p(\Nc)}
\]
for every $\mi'\in \cM_+(E\times F_\C)$ with $\rho(\supp{\mi'})\subseteq \overline B_F(0,R')$, and for every $f\in \Bc^p_K(\Nc)$.  
Then, 
\[
\norm{(1-\chi_{ B((\zeta,z),R''+1)}) f}_{L^p(\mi)}\meg C_1 \norm{M_{1}(\mi)}_{L^\infty(E\times F_\C)}^{1/p} \sup_{\abs{h}\meg R'+1}\norm{[(1-\chi_{ B((\zeta,z),R'')}) f]_h}_{L^p(\Nc)}
\]
for every $R''>0$, for every $(\zeta,z)\in E\times F_\C$, and for every $f\in \Bc^p_K(\Nc)$. 
Furthermore, by assumption, there is a constant $C'>0$ such that
\[
\norm{f}_{\Bc^p_K(\Nc)}\meg C' \norm{f}_{L^p(\mi)}
\]
for every $f\in \Bc^p_K(\Nc)$.

Then, take $f\in \Bc^p_K(\Nc)$ so that $\norm{f}_{\Bc^p_K(\Nc)}=1$ (cf.~\cite[Proposition 3.4]{Bernstein}), and define $f^{(\zeta,x)}\coloneqq  L_{(\zeta,x+i\Phi(\zeta))} f$ for every $(\zeta,x)\in \Nc$. Define $C''\coloneqq \sup_{\abs{h}\meg R'}\norm{f_h}_{L^\infty(\Nc)}$, so that $C''<\infty$ by~\cite[Theorem 3.2]{Bernstein}.
Therefore,
\[
\begin{split}
 C''^p [M_{R''+1}(\mi)]_0(\zeta,x)&\Meg  \int_{B((\zeta,x+i\Phi(\zeta)),R''+1)} \abs{f^{(\zeta,x)}}^p\,\dd \mi\\
 &=\norm{f^{(\zeta,x)}}_{L^p(\mi)}^p-\norm{(1-\chi_{B((\zeta,x+i \Phi(\zeta)), R''+1)}) f^{(\zeta,x)}}_{L^p(\mi)}^p\\
 &\Meg \frac{1}{C'^p} -C_1^p \norm{M_{1}(\mi)}_{L^\infty(\Nc)} \sup_{\abs{h}\meg R'+1} \norm{[(1-\chi_{ B((\zeta,x+i \Phi(\zeta)),R'')})  f^{(\zeta,x)}]_h}_{L^p(\Nc)}^p\\
 &=\frac{1}{C'^p} -C_1^p \norm{M_{1}(\mi)}_{L^\infty(\Nc)} \sup_{\abs{h}\meg R'+1}\norm{[(1-\chi_{ B((0,0),R'')})  f]_h}_{L^p(\Nc)}^p
\end{split}
\]
for every $R''>0$ and for every $(\zeta,x)\in \Nc$. Then, choose $R''>0$ so that\footnote{Notice that this is possible since the set of $f_h$, $h\in \overline B_F(0,R'+1)$ is \emph{compact} in $L^p(\Nc)$, as the mapping $F\ni h\mapsto f_h\in L^p(\Nc)$ is continuous.} 
\[
\sup_{\abs{h}\meg R'+1}\norm{[(1-\chi_{ B((0,0),R'')})  f]_h}_{L^p(\Nc)}< \frac{1}{C' C_1\norm{M_{1}(\mi)}_{L^\infty(\Nc)}^{1/p}},
\]
whence the result.
\end{proof}

We now provide an abstract, yet quite powerful, sufficient condition for sampling measures, based on~\cite[Theorem 5]{Luecking4} (cf.~also~\cite[Theorem 7.9]{CP2}). We shall then draw several corollaries.

\begin{deff}
Take $\mi \in \cM_+(E\times F_\C)$. Then, we define $W(\mi)$ as the closure of $\Set{L_{(\zeta,x+i \Phi(\zeta))}\mi\colon (\zeta,x)\in \Nc}$ in the vague topology.
\end{deff}

\begin{teo}\label{teo:1}
Take a compact subset $K$ of $\overline{\Lambda_+}$, $p,q\in (0,\infty)$ with $q\meg p$,  $\eps>0$, and a $p$-Carleson measure $\mi$ for $\Bc^p_K(\Nc)$ such that $\rho(\supp{\mi})$ is bounded. Assume that the support of every element of $W(\mi)$ is a set of uniqueness for $\Bc^q_{K_\eps}(\Nc) $. 
Then, $\mi$ is a strongly $p$-sampling measure for $\Bc^p_K(\Nc)$.
\end{teo}

Before we pass to the proof, we need some lemmas.

\begin{lem}\label{lem:3}
Let $\Ms$ be a subset of $\cM_+(E\times F_\C)$ such that
\[
\bigcup_{\mi \in \Ms} \rho(\supp{\mi})
\]
is bounded, and
\[
\sup_{\mi \in \Ms} \norm{M_{1}(\mi)}_{L^{\infty}(E\times F_\C)}<\infty.
\]
Then,
\[
\lim_{\mi,\Ff} \norm{f}_{L^p(\mi)}=\norm{f}_{L^p(\mi_0)}
\]
for every filter $\Ff$ on $\Ms$ which converges vaguely to some Radon measure $\mi_0$ on $E\times F_\C$,  for every $f\in \Bc^p_K(\Nc)$, for every compact subset $K$ of $\overline{\Lambda_+}$, and for every $p\in (0,\infty)$.
\end{lem}

The proof is based on~\cite[Theorem 1]{Luecking4} (cf.~also~\cite[Lemma 7.10]{CP2}).

\begin{proof}
Fix $p\in (0,\infty)$ and $R'>0$ so that $\rho(\supp{\mi})\subseteq \overline B_F(0,R')$ for every $\mi \in \Ms\cup \Set{\mi_0}$.
By Lemma~\ref{lem:1} and Proposition~\ref{prop:11bis}, there is a constant $C>0$ such that
\[
\norm{(1-\chi_{ B((0,0),R+1)}) f}_{L^p(\mi)}\meg C \sup_{\abs{h}\meg R'+1} \norm{[(1-\chi_{ B((0,0),R)}) f]_h}_{L^p(\Nc)}
\]
for every $\mi\in \Ms\cup\Set{\mi_0}$, for every $f\in \Bc^p_K(\Nc)$, and for every $R>0$ (cf.~the proof of Proposition~\ref{prop:12}). 
Then, fix $f\in \Bc^p_K(\Nc)$, and take $\eps>0$. Fix $R>1$ so that $ \norm{[(1-\chi_{ B((0,0),R)}) f]_h}_{L^p(\Nc)}\meg\eps$ for every $h\in \overline B_F(0,R'+1)$ (cf.~the proof of Proposition~\ref{prop:12}), and choose $\tau\in C^\infty(E\times F_\C)$ so that $\chi_{B((0,0),R+1)}\meg \tau \meg \chi_{B((0,0),R+2)}$. Then, 
\[
\begin{split}
 \limsup_{\mi, \Ff} \abs*{\norm{f}_{L^p(\mi)}^p-\norm{f}_{L^p(\mi_0)}^{p}}\meg  (C\eps)^p+\limsup_{\mi, \Ff} \abs*{\norm{\tau f}_{L^p(\mi)}^{p}-\norm{\tau f}_{L^p(\mi_0)}^{p}}=\eps
\end{split}
\]
for every $\eps>0$, whence the result.
\end{proof}

\begin{lem}\label{lem:6}
Take $p\in (0,\infty)$, two  compact subsets $K,K'$ of $\overline{\Lambda_+}$, $g\in \Hol_{K'}(E\times F_\C)$ with $g(0,0)\neq 0$, and $\mi \in \cM_+(\Nc)$ such that $\rho(\supp{\mi})$ is bounded in $F$, $M_{1}(\mi)$ is bounded, and such that the support of every element of $W(\mi)$ is a set of uniqueness for $\Bc^q_{K+K'}(\Nc)$. 
Take $\eps>0$, and define, for every $(\zeta,x)\in \Nc$,
\[
U_\eps(\zeta,x)\coloneqq \Set{f\in \Hol_K(E\times F_\C)\colon  \abs{f_0(\zeta,x)}\Meg\eps \norm{f_0 L_{(\zeta,x)}g_0}_{L^p(\Nc)}}.
\]

Then, there is  a constant $C>0$ such that
\[
\norm{f L_{(\zeta,x+i\Phi(\zeta))}g }_{L^{p}(\mi')}\Meg C\norm{f_0 L_{(\zeta,x)} g_0}_{L^p(\Nc)}
\]
for every $(\zeta,x)\in \Nc$, for every $f\in U_\eps(\zeta,x)$ and for every $\mi'\in W(\mi)$.
\end{lem}

The proof is based on~\cite[Lemma 4]{Luecking4} (cf.~also~\cite[Lemma 7.11]{CP2}).

\begin{proof}
By the left invariance of $W(\mi)$, we may reduce to proving the assertion for $(\zeta,x)=(0,0)$.
Define
\[
\mi_{(\zeta,x)}\coloneqq L_{(\zeta,x+i\Phi(\zeta))} \mi
\]
for every $(\zeta,x)\in \Nc$.
Observe that Lemma~\ref{lem:3} implies that it will suffice to prove the assertion with $W(\mi)$ replaced by $\Set{\mi_{(\zeta,x)}\colon (\zeta,x)\in \Nc}$.
Then assume, by contradiction, that there are a sequence $(f^{(j)})$ of elements of $U_\eps(0,0)$, and a sequence $(\zeta_j,x_j)_j$ of elements of $\Nc$ such that 
\[
\norm{f^{(j)}_0 g_0}_{L^p(\Nc)}=1
\]
for every $j\in \N$, while
\[
\lim_{j\to \infty }\norm{f^{(j)} g  }_{L^{p}(\mi_{(\zeta_j,x_j)})}=0.
\]
Observe that 
\[
\norm{M_{1}(\mi_{(\zeta,x)})}_{L^\infty(\Nc)}= \norm{M_{1}(\mi)}_{L^\infty(\Nc)}
\]
for every $(\zeta,x)\in \Nc$, so that $W(\mi)$ is bounded, hence compact and
metrizable, in the vague topology. Therefore, we may assume that
$(\mi_{(\zeta_j,x_j)})  $ converges vaguely to some (positive Radon)
measure $\mi'$ on $D$. Analogously, since $f^{(j)} g$ is bounded in $\Bc_{K+K'}^p(\Nc)$, and since $\Bc^p_{K+K'}(\Nc)$ embeds continuously into $\Hol(E\times F_\C)$ (cf.~\cite[Corollary 3.3]{Bernstein}), we may assume that $(f^{(j)} g)$ converges locally uniformly to some $h\in \Bc_{K+K'}^p(\Nc)$, so
that $\abs{h(0,0)}\Meg \eps \abs{g(0,0)}>0$. Let $(\psi_k)_{k\in K}$ be a
partition of the unity on $E\times F_\C$ whose elements belong to $C_c(\Nc)$, and
observe that 
\[
\lim_{j\to \infty }\norm{\psi_k^{1/p} f^{(j)} g }_{L^{p}(\mi_{(\zeta_j,x_j)})}= \norm{\psi_k^{1/p} h}_{L^{p}(\mi')}
\]
by the previous remarks. Therefore, by Fatou's lemma,
\[
\begin{split}
 0&= \lim_{j\to \infty }\norm{f^{(j)} g }_{L^{p}(\mi_{(\zeta_j,x_j)})}^{p}\\
 &= \lim_{j\to \infty }\sum_{k\in K}\norm{\psi_k^{1/p} f^{(j)} g }_{L^{p}(\mi_{(\zeta_j,x_j)})}^{p}\\
 &\Meg \sum_{k\in K}\norm{\psi_k^{1/p} h }_{L^{p}(\mi')}^{p}\\
 &= \norm{h }_{L^{p}(\mi')}^{p}.
\end{split}
\]
Since the support of $\mi'\in W(\mi)$ is a set of uniqueness for $\Bc^p_{K+K'}(\Nc)$, this implies that $h=0$, which is absurd, since $\abs{h(0,0)}\Meg \eps \abs{g(0,0)}>0$. 
\end{proof}

\begin{lem}\label{lem:7}
Take $p,q\in (0,\infty)$, with $q\meg p$, two compact  subsets $K,K'$ of $\overline{\Lambda_+}$, and $g\in \Oc_{K'}(\Nc)\cap L^q(\Nc)$. 
For every $\eps>0$ and for every $f\in \Oc_K(\Nc)$, define
\[
B_{\eps,f}^{q,g}\coloneqq \Set{ (\zeta,x)\in \Nc\colon \abs{f(\zeta,x)}\meg \eps  \norm{f L_{(\zeta,x)} g}_{L^q(\Nc)}}.
\] 
Then, there is a constant $C>0$ such that
\[
\norm{\chi_{B_{\eps,f}^{q,g}} f}_{L^p(\Nc)}\meg C\eps \norm{f}_{L^p(\Nc)}
\]
for every $\eps>0$ and for every $f\in \Oc_K(\Nc)\cap L^p(\Nc)$.
\end{lem}

The proof is based on~\cite[Lemma 2]{Luecking4} (cf.~also~\cite[Lemma 7.13]{CP2}).

\begin{proof}
It will suffice to observe that the operator
\[
T\colon f \mapsto \left[(\zeta,x)\mapsto \int_\Nc f(\zeta',x') \abs{L_{(\zeta,x)} g(\zeta',x')}^q\,\dd(\zeta',x')   \right] = f* \abs{g^*}^q
\]
induces a continuous linear mapping  of $L^{p/q}(\Nc)$ into itself, which is a consequence of Young's inequality. 
\end{proof}

\begin{proof}[Proof of Theorem~\ref{teo:1}.]
 Notice first that, because of Lemma~\ref{lem:2}, it will suffice to prove that $\mi$ is $p$-sampling for $\Bc^p_K(\Nc)$.
 
Set  $K'\coloneqq \overline{\Lambda_+}\cap \overline B_{F'}(0,\eps)$. Fix $g\in \Bc^q_{K'}(E\times F_\C)$ such that $g(0,0)\neq 0$ (cf.~\cite[Proposition 3.4]{Bernstein}).
Observe that Lemma~\ref{lem:6} implies that for every $\eps'>0$ there is a constant $C'_{\eps'}>0$ such that
\[
\norm{f L_{(\zeta,x+i\Phi(\zeta))}g
}_{L^{q}(\mi)}\Meg C'_{\eps'}\norm{f_0 L_{(\zeta,x)}g_0}_{L^q(\Nc)}
\]
for every $(\zeta,x)\in \Nc$ and for every $f\in \Hol_K(E\times F_\C)$  such that
\[
\abs{f(\zeta,x)}>\eps' \norm{f L_{(\zeta,x)} g}_{L^q(\Nc)},
\]
that is, for every $f\in \Hol_K(E\times F_\C)$ and for every $(\zeta,x)\in \Nc\setminus B_{f_0,\eps'}^{q,g_0}$, with the notation of Lemma~\ref{lem:7}.
In addition, by~\cite[Theorem 3.2]{Bernstein}, there is a constant $C''>0$ such that
\[
\norm{h_0}_{L^\infty(\Nc)}\meg C''\norm{h_0}_{L^q(\Nc)}
\]
for every $h\in \Oc_{K+K'}(\Nc)$, so that
\[
\norm{f L_{(\zeta,x+i\Phi(\zeta))} g }_{L^{q}(\mi)}\Meg \frac{C'_{\eps'} \abs{g(0,0)}}{C''} \abs{f(\zeta,x)}
\]
for every $f\in \Hol_K(E\times F_\C)$ and for every $(\zeta,x)\in \Nc\setminus B_{f_0,\eps'}^{q,g_0}$. By Lemma~\ref{lem:7}, we may choose $\eps'$ so small that
\[
\norm{h}_{L^p(\Nc)}\meg 2 \norm{(1-\chi_{ B_{f,\eps'}^{q,g}})h}_{L^p(\Nc)}
\]
for every $h\in \Oc_K(\Nc)\cap L^p(\Nc)$.
Therefore,
\[
\norm{f_0 }_{L^p(\Nc)}\meg \frac{2 C''}{C'_{\eps'} \abs{g(0,0)}} \norm*{(\zeta,x)\mapsto \norm{f L_{(\zeta,x+i\Phi(\zeta))}g }_{L^{q}(\mi)}  }_{L^p(\Nc)}
\]
for every $f\in \Bc_K^p(\Nc)$.  Hence, it will suffice to show that the linear mapping
\[
T\colon f\mapsto \left[(\zeta,x)\mapsto  \int_{E\times F_\C}f \abs{ L_{(\zeta,x+i\Phi(\zeta))} g }^q\,\dd \mi  \right]
\]
maps $L^{p /q}(\mi)$ continuously into $L^{p/q}(\Nc)$. To see this, observe that, by Jensen's inequality,
\[
\abs*{ \int_{E\times F_\C} f \abs{ L_{(\zeta,x+i\Phi(\zeta))} g }^q\,\dd \mi }^{p/q}\meg \norm{L_{(\zeta,x+i\Phi(\zeta))}g}_{L^q(\mi)}^{p-q} \int_{E\times F_\C} \abs{f}^{p/q} \abs{ L_{(\zeta,x+i\Phi(\zeta))} g }^q\,\dd \mi 
\]
for every $(\zeta,x)\in \Nc$ and for every $f\in L^{p/q}(\mi)$. Observe that,  by Proposition~\ref{prop:11bis}, there is a constant $C'''>0$ such that
\[
\norm{h}_{L^q(\mi)}\meg C''' \norm{h_0}_{L^q(\Nc)}
\]
for every $h\in \Bc_{K'}^q(\Nc)$, so that the preceding remarks show that
\[
\begin{split}
 \norm*{ \int_{E\times F_\C} f \abs{ L_{(\zeta,x+i\Phi(\zeta))} g }^q\,\dd \mi }_{L^{p/q}(\Nc)}^{p/q}&\meg C'''^{p-q} \norm{g_0}_{L^q(\Nc)}^{p-q}  \int_{E\times F_\C} \abs{f}^{p/q} \int_\Nc\abs{ L_{(\zeta,x+i\Phi(\zeta))} g }^q\,\dd (\zeta,x)\,\dd \mi \\
 &\meg C'''^{p-q} \norm{g_0}_{L^q(\Nc)}^{p-q}\sup_{h\in \rho(\supp{\mi})} \norm{g_h}_{L^q(\Nc)}^{q} \norm{f}_{L^{p/q}(\mi)}^{p/q} 
\end{split}
\]
for every $f\in L^{p/q}(\mi)$. The proof is complete.
\end{proof}

\begin{cor}\label{cor:2}
 Take $p\in (0,\infty)$, a compact subset $K$ of $\overline{\Lambda_+}$, $\nu\in \cM_+(F)$ with compact support and a bounded Borel measurable function $g$  on $E\times F_\C$. Assume that $\Bc^p_K(\Nc)\neq \Set{0}$ and define a Radon measure $\mi_{\nu,g}$ so that 
 \[
 \int_{E\times F_\C} \varphi \,\dd \mi_{\nu,g} = \int_F \int_\Nc g_h \varphi_h\,\dd \Hc^{2n+m} \,\dd \nu(h)
 \]
 for every $\varphi\in C_c(E\times F_\C)$.
 Then, $\mi_{\nu,g}$ is a (actually,  strongly) $p$-sampling measure for $\Bc^p_K(\Nc)$  if and only if  there are $R,C>0$ such that $[M_{R}(\mi_{\nu,g})]_0\Meg C$ on $ \Nc$.
\end{cor}

Notice that, in particular, this result characterizes the dominant subsets of $\Nc$ for the space $\Oc_K(\Nc)\cap L^p(\Nc)$\footnote{That is, the subsets $D$ of $\Nc$ such that $\chi_D\cdot \Hc^{2n+m}$ is a $p$-sampling measure for $\Oc_K(\Nc)\cap L^p(\Nc)$.} (choosing $\nu=\delta_0$ and $g=1$), and the dominant subsets of $\rho^{-1}(\overline B_F(0,R))$ for the space $\Bc^p_K(\Nc)$\footnote{That is, the subsets $D$ of $\rho^{-1}(\overline B_F(0,R))$ such that $\chi_D\cdot \Hc^{2n+2m}$ is a $p$-sampling measure for $\Bc_K^p(\Nc)$.} (choosing $\nu=\chi_{\overline B_F(0,R)}\cdot \Hc^m$ and $g=1$).

\begin{proof}
 One implication follows from Proposition~\ref{prop:12}, since $\rho(\supp{\mi_{\nu,g}}) \subseteq \supp{\nu}$ and 
 \[
 M_{1}(\mi_{\nu,g})\meg \norm{g}_{L^\infty(E\times F_\C)} \nu(F) \Hc^{2n+m}(B_\Nc((0,0),1)).
 \]
 Conversely, assume that there are $R,C>0$ such that $[M_{R}(\mi_{\nu,g})]_0\Meg C$ on $\Nc$. Take $\mi'\in W(\mi_{\nu,g})$, and observe that $\mi'\meg \norm{g}_{L^\infty(E\times F_\C)}\mi_{\nu,1}$, so that $\mi'$ is absolutely continuous with respect to $\mi_{\nu,1}$. In particular, either $\mi'=0$ or $\mi_{\nu,1}(\supp{\mi'})>0$, in which case $\Hc^{2n+m}(\rho^{-1}(h)\cap \supp{\mi'})>0$ for some $h\in F$, so that $\supp{\mi'}$ is a set of uniqueness for $\Hol(E\times F_\C)$. 
 By Theorem~\ref{teo:1}, it will then suffice to prove that $\mi'\neq 0$. Nonetheless, if $\varphi\in C_c(E\times F_\C)$ and $\varphi \Meg\chi_{B((0,0),R)}$, then 
 \[
 \int_{E\times F_\C} \varphi\,\dd L_{(\zeta,x)}\mi_{\nu,g}\Meg M_{R}(\mi_{\nu,g})(\zeta,x)\Meg C
 \]
 for every $(\zeta,x)\in \Nc$, so that $\int_{E\times F_\C} \varphi\,\dd \mi'\Meg C$. The proof is complete.
\end{proof}

\begin{cor}\label{cor:3}
 Take $p\in (0,\infty)$, $\eps,C>0$, and a compact subset $K$ of $\overline{\Lambda_+}$. In addition, fix $\nu\in \cM_+(E\times F_\C)$.
 For every $\mi\in \cM_+(E\times F_\C)$ and for every $R>0$, define
 \[
 G_{\mi,R}\coloneqq \Set{(\zeta,z)\in E\times F_\C\colon M_R(\mi)(\zeta,z)\Meg \eps}.
 \]
 Then, there is $R>0$ such that every $R'>0$ and for every $\mi\in \cM_+(\Nc)$ such that
 \[
 M_{1}(\mi)\in L^\infty(\Nc)
 \]
 and
 \[
 \nu(G_{\mi,R}\cap B((\zeta,x+i\Phi(\zeta)),R))\Meg C,
 \]
 for every $(\zeta,x)\in \Nc\setminus B((0,0),R')$,  is a strongly $p$-sampling measure for $\Bc^p_K(\Nc)$.
\end{cor}

This result in inspired by~\cite[Theorem 4.2]{Luecking3} (cf.~also~\cite[Theorem 7.6]{CP2}). The proof is different, though. Cf.~also~\cite[Theorem 5]{Lin}.

\begin{proof}
 We first prove the assertion for $R'=0$.
 
 For every $R>0$, fix $\varphi_R \in C_c(\Nc)$ such that $\chi_{B((0,0),R)}\meg  \varphi\meg \chi_{B((0,0),2R)}$.
 Define, for every $\mi\in \cM_+(E\times F_\C)$ and for every $R>0$,
 \[
 M'_{R}(\mi)\colon E\times F_\C\ni (\zeta,z)\mapsto \int_{E\times F_\C} L_{(\zeta,z)} \varphi_R\,\dd \mi
 \]
 and
 \[
 G'_{\mi,R}\coloneqq \Set{(\zeta,z)\in E\times F_\C\colon M'_R(\mi)(\zeta,z)\Meg \eps},
 \]
 so that
 \[
 M_R(\mi)\meg M'_R(\mi)\meg M_{2 R}(\mi) \qquad \text{and} \qquad
 G_{\mi,R}\subseteq G'_{\mi,R}\subseteq G_{\mi,2 R}.
 \]
 
 By Proposition~\ref{prop:11bis} and~\cite[Theorem 1.16]{Bernstein}, we may find $R>0$ such that, for every $(R,6)$-lattice $(\zeta_j,z_j)_{j\in J}$ on $\rho^{-1}(B_F(0,2R))$, the mapping
 \[
 f \mapsto (f(\zeta_j,z_j))_j
 \]
 induces an isomorphism of $\Bc_{K+\overline B_{F'}(0,1)\cap \overline{\Lambda_+}}^p(\Nc)$ onto a closed subspace of $\ell^p(J)$.\footnote{Let $(\zeta'_k,x'_k)_{k\in K}$ be a $(R,2)$-lattice on $\Nc$. Then, we may define $\iota\colon J\to K$ so that $d((\zeta_j,z_j), (\zeta'_{\iota(j)},x'_{\iota(j)}+i\Phi(\zeta'_{\iota(j)})))<4R$ for every $j\in J$. Clearly, there is a dimensional constant $N\in \N$ such that the fibres of $\iota$ have at most $N$ elements. Therefore, using~\cite[Lemma 3.25]{CalziPeloso} and Proposition~\ref{prop:11bis}, we see that there is a dimensional constant $C>0$ such that $\norm{(f(\zeta_j,z_j))_j}_{\ell^p(J)}\meg C  \norm{(f_0(\zeta'_k,x'_k))_k}_{\ell^p(K)}+CR\norm{f_0}_{L^p(\Nc)}$ for every $f\in \Bc^p_{K+\overline B_{F'}(0,1)\cap \overline{\Lambda_+}}$, provided that (say) $R\meg 1$. The assertion then follows from Proposition~\ref{prop:11bis} and~\cite[Theorem 1.16]{Bernstein}.  }
 
 Then, take $\mi\in \cM_+(E\times F_\C)$ so that $M_{1}(\mi)\in L^\infty(E\times F_\C)$, and assume that 
 \[
 C'\meg \nu(G_{\mi,R}\cap B((\zeta,x+i\Phi(\zeta)),R))\meg \nu(G'_{\mi,R}\cap B((\zeta,x+i\Phi(\zeta)),R))
 \]
 for every $(\zeta,x)\in \Nc$. Observe that, since $M_{1}(\mi)\in L^\infty(\Nc)$, the set of $L_{(\zeta,x+i\Phi(\zeta))}\mi$, as $(\zeta,x)$ runs through $\Nc$, is bounded in $\cM_+(E\times F_\C)$, hence relatively compact and metrizable. In particular, if $\mi'\in W(\mi)$, then there is a sequence $(\zeta_j,x_j)_j$ of elements of $\Nc$ such that $L_{(\zeta_j,x_j+i\Phi(\zeta_j))}\mi$ converges vaguely to $\mi'$. Hence,  $M'_R(L_{(\zeta_j,x_j+i\Phi(\zeta_j))}\mi)$ converges locally uniformly to $M'_R(\mi')$, so that
 \[
 \bigcap_{k\in \N} \bigcup_{j\Meg k} G'_{L_{(\zeta_j,x_j+i\Phi(\zeta_j))}\mi,R}\subseteq G'_{\mi',R},
 \]
 whence
 \[
 \nu(G'_{\mi',R}\cap B((\zeta,x+i\Phi(\zeta)),R))\Meg \lim_{k\to \infty} \nu\bigg( \bigcup_{j\Meg k}G'_{L_{(\zeta_j,x_j)}\mi,R}\cap B((\zeta,x+i\Phi(\zeta)),R)\bigg)\Meg C'
 \]
 for every $(\zeta,x)\in \Nc$. Therefore, $\supp{\mi'}\cap B((\zeta,x+i\Phi(\zeta)),2R)\neq \emptyset$  for every $(\zeta,x)\in \Nc$, so that we may find a $(R,6)$-lattice on $\rho^{-1}(B_F(0,2R))$ whose elements belong to $\supp{\mi'}$.\footnote{It suffices to take a family $(\zeta_j,x_j)_j$ of elements of $\supp{\mi'}$ which is maximal for the relation $d((\zeta_j,x_j),(\zeta_k,x_k))\Meg 2R$.} Our choice of $R$ then implies that every element of $\Bc^p_{K+\overline B_{F'}(0,1)\cap\overline{\Lambda_+}}(\Nc)$ which vanishes on $\supp{\mi'}$ must vanish identically. Therefore, Theorem~\ref{teo:1} implies that $\mi$ is a strongly $p$-sampling measure for $\Bc_K^p(\Nc)$. The assertion for $R'>0$ follows from Proposition~\ref{prop:4}, applying the preceding arguments to $\mi+\chi_{B((0,0),2R)}\cdot \Hc^{2n+2m}$.
\end{proof}

\begin{cor}\label{cor:4}
 Take $p\in (0,\infty)$, $r<\frac \pi 2$, $\lambda\in F'$, and a compact  subset $K$ of $\overline{\Lambda_+}$ with a non-empty interior. Fix a basis $(e_k)$ of $E$ with dual basis $(e'_k)$, and assume that $H_K(-\Phi(\zeta))<\frac\pi 2 \sum_{k} \abs{\langle e'_k,\zeta\rangle}^2$ for every $\zeta\in E\setminus \Set{0}$. Define $J\coloneqq \Z[i]^{n}$ and $\zeta_j\coloneqq \sum_k j_k e_k$ for every $j\in J$. For every $j\in J$, take a sequence $(x_{j,j'})_{j'\in \N}$ of elements of $F$ which is a finite union of separated families,\footnote{We say that a family $(x'_{j''})$ of elements of $F$ is uniformly separated if $\inf_{j''_1\neq j''_2} \abs{x'_{j''_1}-x'_{j''_2}}>0$. This condition is therefore equivalent to saying that $\sum_{j'}\delta_{x_{j,j'}}$ is a $p$-Carleson measure for $\Bc^p_K(F)$, thanks to Corollary~\ref{cor:1}.} and assume that $\sum_{j'\in \N} (x_{j,j'}+r [(K-\lambda)\cup (\lambda-K)]^\circ)=F$.
 Then, $(\zeta_j,x_{j,j'}+i\Phi(\zeta_j))_{j\in J,j'\in \N}$ is a strongly sampling family for $\Bc^p_K(\Nc)$.
\end{cor}

Observe that this result implies~\cite[Theorem
4.1]{OlevskiiUlanovskii} (of which is a consequence and to which it
essentially reduces when $n=0$) by means of~\cite[Theorem
2.1]{OlevskiiUlanovskii} (or of the following
Proposition~\ref{prop:2}). Notice that~\cite[Theorem
1.6]{MonguzziPelosoSalvatori} gives a slightly stronger result,
corresponding to the case in which $m=1$, $n\Meg 1$, $p=2$, $r=\frac
\pi 2$, $K=[0,a]$, $\lambda=\frac a 2$, $(e_k)$ is orthogonal with
respect to $\Phi$, and $\Set{x_{j,j'}\colon j'\in J'}=\frac{2 }{\pi
  a}\Z$. Notice that this latter result only gives rise to sampling
families (and not strongly sampling families). Actually, the above
families cannot be (in general) strongly sampling if $r=\frac \pi 2 $,
at least when $n=0$ and  $K$ is a parallelotope (cf.~\cite[Nos.\ 34,
44]{PlancherelPolya}). 

Notice that, by means of Proposition~\ref{prop:4}, we may drop a finite number of terms from the above families without compromising the strong sampling property.

Corollary~\ref{cor:4} will be a consequence of the more general Corollary~\ref{cor:6}.

\begin{cor}\label{cor:6}.
 Take $p\in (0,\infty)$, $R>0$, $r<\frac{\pi}{2}$, $\lambda\in F'$, a compact subset $K$ of $\overline{\Lambda_+}$ with a non-empty interior, and a basis $(e_k)$ of $E$ with dual basis $(e'_k)$. Assume that:
 \begin{itemize}
  \item there are subharmonic functions $(\varphi_k)$ on $\C$, with $m\meg \Delta \varphi_k\meg M$ for some $m,M>0$, such that $H_K(-\Phi(\zeta))\meg \sum_k  \varphi_k(\langle e'_k,\zeta\rangle)$ for every $\zeta\in $;
  
  \item for every $k$ there is a family $(\zeta_{k,j})_{j\in \N}$ of elements of $\C$ which is a finite union of uniformly separated sequences, and which contains a uniformly separated sequence $(\zeta'_{k,j})_{j\in \N}$ such that\footnote{By~\cite[Theorem 1]{OS}, this is equivalent to saying that $(\zeta_{k,j})_{j\in \N}$ is a sampling sequence for the Fock space $\Set{f\in \Hol(\C)\colon  \int_\C \abs{f(w)}^2 \ee^{-\varphi(w)}\,\dd w <\infty}$.}
  \[
  \liminf_{R'\to \infty} \inf_{w\in \C}\frac{\card\Big(B_\C(w,R')\cap \Set{\zeta'_{k,j}\colon j\in \N}\Big)}{\int_{B_\C(w,R')} \Delta \varphi_k(w')\,\dd w'}>\frac 2 \pi;
  \]
  
  \item for every $j\in \N^n$ there is a sequence $(x_{j,j'})_{j'\in \N}$ of elements of $F$ which is a finite union of uniformly separated sequences such that
  \[
  \sum_{j'\in \N}  (x_{j,j'}+r[(K-\lambda)\cup(\lambda-K)]^\circ)\supseteq F\setminus B(x_j,R)
  \]
  for some $x_j\in F$.
 \end{itemize}
 Define $\zeta_j\coloneqq \sum_k \zeta_{k,j_k}e_k$ for every  $j\in \N^n$. Then, $(\zeta_j,x_{j,j'}+i\Phi(\zeta_j))_{j\in \N^n,j'\in \N}$ is a strongly sampling family for $\Bc^p_K(\Nc)$.
\end{cor}

In the above statement, $\Delta=\partial_w \overline{\partial_w}$, so that it equals $1/4$ the ordinary Laplacian, and $\Delta(\abs{\,\cdot\,}^2)=1$.

For the sake of simplicity, we shall prove separately the following elementary lemma.

\begin{lem}\label{lem:4}
 Let $X$ be a locally compact space with a countable base, and denote by $\cM_d(X)$ the space of discrete Radon measures $\mi$ on $X$ such that $\mi(\Set{x})\in \Z$ for every $x\in X$. Then, $\cM_d(X)$ is vaguely closed. 
 In addition, if $(\mi_j)$ is a sequence of positive elements of $\cM_d(X)$ which converges vaguely to some $\mi$, then 
 \[
 \supp{\mi}=\Set{x\in X\colon \exists (x_j)\in \prod_j\supp{\mi_j}\:\: \lim_{j\to \infty} x_j= x  }.
 \]
\end{lem}

\begin{proof}
 Observe first that $\cM_d(X)$ is the set of Radon measures $\mi$ on $X$ such that $\mi(U)\in \Z$ for every relatively compact open subset $U$ of $X$. One implication in obvious. Conversely, take $\mi$ as above, and observe that $\mi(E)\in \Z$ and $\abs{\mi}(E)\in \N$ for every $\mi$-integrable subset of $X$.  Observe that $X$ is a Polish space, thanks to~\cite[Corollary to Proposition 16 of Chapter IX, \S\ 2, No.\ 10, and Corollary to Proposition 2 of Chapter IX, \S\ 6, No.\ 1]{BourbakiGT2}, so that we may endow $X$ with a complete metric $d$. Take a compact subset $K$ of $X$ such that $\abs{\mi}(K)>0$, and observe that for every $k\in \N$ there is a finite family $(x_{k,j})$ of elements of $X$ such that $K\subseteq B_X(x_{k,j},2^{-k})$, so that $\abs{\mi}(B_X(x_{k,j_k},2^{-k}))\Meg 1$ for some $j_k$. By induction, we may assume that $x_{k,j_k}\in B_X(x_{k',j_{k'}},2^{-k'}) $ for every $k'<k$. Thus,  $(x_{k,j_k})$ is a Cauchy sequence, so that it converges to some $x\in K$ with $\abs{\mi}(\Set{x})=\lim\limits_{k\to \infty}\abs{\mi}(B_X(x_{k,j_k},2^{-k}))\Meg 1$. In particular, there is $k\in \N$ such that $\abs{\mi}(B_X(x_{k',j_{k'}},2^{-k'}))=\abs{\mi}(\Set{x})$ for every $k'\Meg k$. Since $K$ was arbitrary, applying the same argument to $K\setminus B_X(x_{k,j_k},2^{-k})$ with $k$ as before, we infer that the restriction of $\abs{\mi}$ to $K$ belongs to $\cM_d(K)$. By the arbutrariness of $K$, this implies that $\abs{\mi}\in \cM_d(X)$, so that $\mi\in \cM_d(X)$.
 
 Now, let $\Ff$ be a filter on $\cM_d(X)$ converging vaguely to some $\mi$. Let $U$ be a relatively compact open subset of $X$, and take $\varphi\in C_c(X)$ so that $\varphi=1$ on $U$. Then,~\cite[Proposition 22 of Chapter IV, \S\ 5, No.\ 12]{BourbakiInt1}, applied to $\varphi\cdot \Fc$ and $\varphi\cdot \mi$, shows that $\mi'(U)\to \mi(U)\in \Z$, as $\mi'$ runs along $\Ff$, provided that $\mi(\partial U)=0$. In particular, observe that there is only a countable number of $\delta>0$ such that $\mi(\partial U_\delta)>0$, where $U_\delta\coloneqq \Set{x\in U\colon d(x, X\setminus U)>\delta}$, so that $\mi(U)=\lim\limits_{\delta\to 0^+ }\mi(U_\delta)\in \Z$. Thus, the preceding remarks imply that $\mi\in \cM_d(X)$.

 Now, let $(\mi_j)$ be a sequence of positive elements of $\cM_d(X)$ which converges vaguely to some $\mi$. Take $x\in \supp{\mi}$, take $r>0$ so that $B_X(x,2 r)$ is relatively compact and $\mi=\delta_x$ on $B_X(x,2 r)$, and observe that, for every $k\in \N$, $\mi_j(B_X(x,2^{-k}r))\to \mi(B_X(x,2^{-k}r))=\mi (\Set{x})>0$ by~\cite[Proposition 22 of Chapter IV, \S\ 5, No.\ 12]{BourbakiInt1} again, so that there is $j_k\in \N$ such that $x_{j,k}\in B_X(x,2^{-k})\cap\Supp{\mi_j}$ if $j\Meg j_k$.  Thus, we may find a sequence $(x'_j)\in \prod_j \supp{\mi_j}$ which converges to $x$.  Conversely, if there is a sequence $(x_j)\in \prod_j \supp{\mi_j}$ which converges to some $x$ in $X$, then clearly $\mi_j(B_X(x,\eps))\Meg 1$ for every $\eps>0$, if $j$ is sufficiently large, so that $\mi(B_X(x,\eps))\Meg 1$ for every $\eps>0$, whence $\mi(\Set{x})\Meg 1$ and $x\in \supp{\mi}$.
\end{proof}

\begin{proof}[Proof of Corollary~\ref{cor:6}.]
 Notice that the assumptions are weaker (while the conclusion is stronger) if we replace $K$ with its convex envelope. Therefore, we may assume that $K$ is convex.
 
 \textsc{Step I.} Assume first that $E=\Set{0}$, write $x_{j'}$ instead of $x_{0,j'}$ and let us prove that $\Set{x_{j'}\colon j'\in \N}$ is a set of uniqueness of $\Bc^\infty_{K}(F)$. Observe that, since $\Bc^\infty_{K-\lambda}(F)= \ee^{-i\langle \lambda_\C,\,\cdot\,\rangle}\Bc^\infty_K(F)$, we may assume that $\lambda=0$, that is, that the convex envelope of  $K\cup (-K)$ is a neighbourhood of $0$.  Then, the assertion follows from~\cite[Theorem 4.1]{OlevskiiUlanovskii} and Proposition~\ref{prop:4}.

 \textsc{Step II.} Let us prove that $\Set{(\zeta_j,x_{j,j'}+i\Phi(\zeta_j))\colon j\in \N^n,j'\in \N}$ is a set of uniqueness for $\Bc^\infty_{K}(\Nc)$. 
 Take $f\in \Bc^\infty_K(\Nc)$, and assume that $f(\zeta_j,x_{j,j'}+i\Phi(\zeta_j))=0$ for every $j\in \N^n$ and for every $j'\in \N$.  Observe first that, for every $j\in \N^n$, if we define $f^{(j)}\coloneqq f(\zeta_j,\,\cdot\,+i \Phi(\zeta_j))$, then $f^{(j)}\in \Bc^p_K(F)$, so that \textsc{step I} implies that $f^{(j)}=0$. Therefore, $f(\zeta_j,z)=0$ for every $z\in F_\C$ and for every $j\in J$. Observe that, since $m\meg \Delta \varphi_k\meg M$ on $\C$ for every $k=1,\dots, n$, there is $\eps>0$ such that
 \[
 \liminf_{R'\to \infty} \inf_{w\in \C}\frac{\card\Big(B_\C(w,R')\cap \Set{\zeta'_{k,j}\colon j\in \N}\Big)}{\int_{B_\C(w,R')} \Delta( \varphi_k+\eps \abs{\,\cdot\,}^2)(w')\,\dd w'}>\frac 2 \pi.
 \]
 Since $f\in \Bc^\infty_K(\Nc)$, this implies that $f(\,\cdot\,,z)$ belongs to the Fock space
 \[
 \Set{g\in \Hol(E)\colon  \int_{E} \abs{g(\zeta)} \ee^{- \sum_{k}( \varphi_k(\langle e'_k,\zeta\rangle)+\eps \abs{\langle e'_k,\zeta\rangle}^2)}\,\dd \zeta<\infty }
 \]
 and vanishes on $\Set{\zeta_j\colon j\in J}$ for every $z\in F_\C$,  so that~\cite[Theorem 1]{OS} (and a simple interation argument, cf.~\cite[Corollary 5.4]{MonguzziPelosoSalvatori}) implies that $f(\,\cdot\,,z)=0$ for every $z\in F_\C$, whence $f=0$.
 
 \textsc{Step III.} Now, define $\mi_{(0,0)}\coloneqq  \sum_{j,j'} \delta_{(\zeta_j,x_{j,j'}+i\Phi(\zeta_j))}$ and
 \[
 \mi_{(\zeta,x)}\coloneqq L_{(\zeta,x+i\Phi(\zeta))} \mi_{(0,0)}
 \]
 for every $(\zeta,x)\in \Nc$.  Observe that $\mi_{(0,0)}$ is a well defined Radon measure on $E\times F_\C$, and that $M_1(\mi_{(0,0)})$ is bounded.
 Take $\mi\in W(\mi_{(0,0)})$, and observe that there is a sequence $(\zeta'_k,x'_k)_{k\in \N}$ of elements of $\Nc$ such that $\mi_{(\zeta'_k,x'_k)}$ converges vaguely to $\mi$ (cf.~the proof of Lemma~\ref{lem:6}). 
 By Lemma~\ref{lem:4}, $\mi$ is a locally finite sum of Dirac deltas, and $\supp{\mi}$ is a suitable limit of $\supp{\mi_{(\zeta'_k,x'_k)}}$. It is not hard to see that $\supp{\mi}=\Set{(\zeta''_j,x''_{j,j'})\colon j\in \N^n, j'\in \N}$, with:
 \begin{itemize}
  \item for every $j\in \N^n$,  $(x_{j,j'})_{j'\in \N}$ is a sequence of elements of $F$ which is a finite union of uniformly separated sequences, and
  \[
  \sum_{j'\in \N}  (x''_{j,j'}+r'[(K_\eps-\lambda)\cup(\lambda-K_\eps)]^\circ)\supseteq F\setminus B(x_j,R)
  \]
  for some $x'''_j\in F$ for some $r'\in (r,\frac \pi 2)$ and some $\eps>0$;\footnote{ It suffices to fix $r'$ and observe that the convex envelope $H$ of $(K-\lambda) \cup (\lambda-K)$ contains $\frac{r}{r'} (H+B_{F'}(0,\eps)) $ for some small $\eps>0$, which contains the convex envelope of $ (K_\eps-\lambda) \cup (\lambda-K_\eps)$. }
  
  \item $\zeta''_j=\sum_{k} \zeta''_{k,j}e_k$, where, for every $k$, $(\zeta''_{k,j})_{j\in \N}$  is a finite union of uniformly separated sequences in $\C$ and  contains a uniformly separated sequence $(\zeta'''_{k,j})_{j\in \N}$ such that (if $\eps$ is sufficiently small)
  \[
  \liminf_{R'\to \infty} \inf_{w\in \C}\frac{\card\Big(B_\C(w,R')\cap \Set{\zeta'''_{k,j}\colon j\in \N}\Big)}{\int_{B_\C(w,R')} \Delta (\varphi_k+\eps\abs{\,\cdot\,}^2)(w')\,\dd w'}>\frac 2 \pi.
  \]
 \end{itemize}
 Take $C>0$ so that $\abs{\Phi(\zeta)}\meg C\sum_k \abs{\langle e'_k,\zeta\rangle}^2$ for every $\zeta\in E$, and observe that $H_{K_{\eps/C}}(-\Phi(\zeta))\meg H_K(-\Phi(\zeta))+\frac\eps C\abs{\Phi(\zeta)}^2\meg \sum_{k} (\varphi_k(\langle e'_k,\zeta\rangle)+\eps\abs{\langle e'_k,\zeta \rangle}^2)$ for every $\zeta\in E$.
 Therefore, \textsc{step II} implies that $\supp{\mi}$ is a set of uniqueness for $\Bc^\infty_{K_{\eps/C}}(\Nc)$, hence for $\Bc^p_{K_{\eps/C}}(\Nc)$. By the arbitrariness of $\mi$ (and using the fact that $\eps$ and $r'$ can be chosen independently of $\mi$), Theorem~\ref{teo:1} leads to the conclusion.
\end{proof}

With a similar (but simpler) argument one may also prove the following result.

\begin{cor}\label{cor:5}
 Take $p\in (0,\infty)$,  a compact subset $K$ of $\overline{\Lambda_+}$, $\delta>0$, $R>1$, and $\eps>0$. Assume that $d$ satisfies the following `convexity' assumption: if $(\zeta,x)\in \Nc$ and $d_\Nc((0,0),(\zeta,x))<1$, then there is $(\zeta',x')\in \Nc$ such that $d_\Nc((0,0),(\zeta',x')), d_\Nc((\zeta',x'),(\zeta,x))<\frac 1 2$.\footnote{  This is the case if $d_\Nc$ is defined as a left-invariant homogeneous control distance. In addition, if $n=0$, then $d_\Nc$ is the distance induced by a norm on $F$, so that this condition is authomatically satisfied for every choice of $d_\Nc$. } If the support of every $(\delta,R)$-lattice on $\Nc$ is a set of uniqueness for $L^q(\Nc)\cap\Oc_{K_\eps}(\Nc)$, then every $(\delta,R)$-lattice on $\Nc$ is strongly sampling for $L^p(\Nc)\cap\Oc_{K_\eps}(\Nc)$.
\end{cor}

\begin{proof}
 The assertion follows from Theorem~\ref{teo:1} and Lemma~\ref{lem:4}, since the condition imposed of $d_\Nc$ guarantees that $B((\zeta,x),\delta)\cap B((\zeta',x'),\delta)=\emptyset$ if and only if $d_\Nc((\zeta,x),(\zeta',x'))\meg 2\delta$.
\end{proof}

\begin{prop}\label{prop:2}
 Take $p\in (0,\infty)$, a compact subset $K$ of $F'$, $\eps>0$, and $\mi \in \cM_+(E\times F_\C)$ with $\rho(\supp{\mi})$ bounded.  Then, the following hold:
 \begin{enumerate}
  \item[\textnormal{(1)}] if $\mi$ is a $p$-sampling measure for $\Bc^p_{K_\eps}(\Nc)$, then the canonical inclusion $\Bc^\infty_K(\Nc)\to L^\infty(\mi)$ is an isomorphism onto its image and $\Bc^\infty_K(\Nc)=\Set{f\in \Hol_K(E\times F_\C)\colon f\in L^\infty(\mi)}$;
  
  \item[\textnormal{(2)}] if $M_K(\mi)$ is bounded, $\mi$ is discrete and $\inf_{\mi(\Set{(\zeta,z)})>0} \mi(\Set{(\zeta,z)})>0$, and the canonical mapping $\Bc^\infty_{K_\eps}(\Nc)\to L^\infty(\mi)$ is an isomorphism onto its image, then $\mi$ is a strongly $p$-sampling measure for $\Bc^p_K(\Nc)$.
 \end{enumerate}  
\end{prop}

This result extends  (i) and (ii) of~\cite[Theorem 2.1]{OlevskiiUlanovskii}.

\begin{proof}
 (1) Assume by contradiction that the canonical mapping  $\Bc^\infty_K(\Nc)\to L^\infty(\mi)$ is not an isomorphism onto its image. Then, there is a sequence $(f^{(j)})$ of elements of $\Bc^\infty_K(\Nc)$ such that $\norm{f^{(j)}_0}_{L^\infty(\Nc)}=1$ and  $\norm{f^{(j)}}_{L^\infty(\mi)}\meg 2^{-j}$ for every $j\in \N$. In particular, for every $j\in \N$ there is $(\zeta_j,x_j)\in \Nc$ such that $\abs{f^{(j)}_0(\zeta_j,x_j)}\Meg \frac 1 2$. 
 Fix $\varphi \in \Bc^p_{\overline B_{F'}(0,\eps)\cap \overline{\Lambda_+}}(\Nc)$ so that $\varphi(0,0)=1$, and define $g^{(j)}\coloneqq f^{(j)} L_{(\zeta_j,x_k+i\Phi(\zeta_j))} \varphi$ for every $j\in \N$. Then, $g^{(j)}\in \Bc^p_{K_\eps}(\Nc)$, 
 \[
 \norm{g^{(j)}}_{L^p(\mi)}\meg 2^{-j}\norm{L_{(\zeta_j,x_k+i\Phi(\zeta_j))} \varphi}_{L^p(\mi)}\meg C2^{-j}\norm{\varphi_0}_{L^p(\Nc)},
 \]
 where $C$ is the norm of the continuous inclusion $\Bc^p_{K_\eps}(\Nc) \subseteq L^p(\mi)$, for every $j\in \N$. Furthermore, since $\Bc^p_{K_\eps}(\Nc)$ embeds continuously into $\Bc^\infty_{K_\eps}(\Nc) $ (cf.~\cite[Theorem 3.2]{Bernstein}), there is a constant $C'>0$ such that
 \[
 \norm{g^{(j)}_0}_{L^p(\Nc)}\Meg C' \norm{g^{(j)}_0}_{L^\infty(\Nc)}\Meg C'\abs{f^{(j)}_0(\zeta_j,x_j)}\Meg \frac{C'}{2} 
 \]
 for every $j\in \N$: contradiction. The second part is proved as Lemma~\ref{lem:2}.
 
 (2) By Proposition~\ref{prop:11bis} we know that $\mi$ is a $p$-Carleson measure for $\Bc^p_K(\Nc)$. Take $g \in    \Bc^p_{\overline B_{F'}(0,\eps)\cap \overline{\Lambda_+}}(\Nc)$ so that $g(0,0)=1$.
 In addition, by assumption there is a constant $C'''>0$ such that
 \[
 \norm{f_0}_{L^\infty(\Nc)}\meg C''\norm{f}_{L^\infty(\mi)}
 \] 
 for every $f\in  \Bc^\infty_{K_\eps}(\Nc)$. Define $C'''\coloneqq \inf_{\mi(\Set{(\zeta,z)})>0} \mi(\Set{(\zeta,z)})^{1/p}$.
 Then, for every $f\in \Bc^p_K(\Nc)$,
 \[
 \begin{split}
  \norm{f_0}_{L^p(\Nc)}&\meg \norm*{(\zeta,x)\mapsto \norm{f_0 L_{(\zeta,x)}g_0}_{L^\infty(\Nc)} }_{L^p(\Nc)}\\
  &\meg C'' \norm*{(\zeta,x)\mapsto \norm{f L_{(\zeta,x+i\Phi(\zeta))}g}_{L^\infty(\mi)} }_{L^p(\Nc)}\\
  &\meg C'' C'''\norm*{(\zeta,x)\mapsto \norm{f L_{(\zeta,x+i\Phi(\zeta))}g}_{L^p(\mi)} }_{L^p(\Nc)}\\
  &\meg C'' C'''  \norm*{(\zeta',z')\mapsto f(\zeta',z')\norm{g_{\rho(\zeta',z')}}_{L^p(\Nc)}}_{L^p(\mi)}\\
  &\meg  C'' C''' \norm{g_0}_{L^p(\Nc)}  \sup_{h\in \rho(\supp{\mi})}\ee^{\eps \abs{h}}\norm{ f}_{L^p(\mi)}
 \end{split}
 \]
 where the second inequality follows from the fact that $ f L_{(\zeta,x+i\Phi(\zeta))}g\in \Bc^\infty_{K_\eps}(\Nc)$ for every $(\zeta,x)\in \Nc$, while the last inequality follows from~\cite[Theorem 1.7]{Bernstein}. Thus, $\mi$ is a $p$-sampling measure for $\Bc^p_K(\Nc)$.  The conclusion follows by means of Lemma~\ref{lem:2}.
\end{proof}

We now show how the general Beurling-type necessary conditions for sampling sequences proved in~\cite{Romeroetal} look like in this setting. Here, for every $\lambda\in \overline{\Lambda_+}$ we denote by $\abs{\Pfaff(\lambda)}$ the (complex) determinant of the positive hermitian form $\langle \lambda, \Phi\rangle$ with respect to the scalar product of $E$.

\begin{prop}\label{prop:3}
 Take a compact subset $K$ of $\overline{\Lambda_+} $, and let $S$ be a locally finite subset of $\Nc$ such that $\mi\coloneqq \sum_{(\zeta,x)\in S} \delta_{(\zeta,x+i\Phi(\zeta))}$ is a sampling (Radon) measure for $\Bc^2_K(\Nc)$. Then, 
 \[
 \liminf_{R\to +\infty} \inf_{(\zeta,x)\in \Nc} \frac{\mi(B_\Nc((\zeta,x),R))}{\Hc^{2n+m}(B_\Nc((\zeta,x),R))}\Meg \frac{2^{n-m}}{\pi^{n+m}} \int_K \abs{\Pfaff(\lambda)}\,\dd \lambda.
 \]
\end{prop}

\begin{proof}
 We may assume that $\Hc^m(K)>0$, that is, $\Bc^2_K(\Nc)\neq \Set{0}$. 
 Then, the result will follow from~\cite[Theorem 2.2]{Romeroetal} and~\cite[Proposition 5.1]{Bernstein} (extended to the case in which $K$ is not necessarily convex), the former applied with $X=\Nc$, $\mi=\Hc^{2n+n}$, and $\Hc=\Oc^2_K(\Nc)$, once we show that the assumptions of~\cite[Section 2.1]{Romeroetal}  are satisfied. The assumption of~\cite[Section 2.1 (A)]{Romeroetal} are clearly satisfied, since the distance $d_\Nc$ is continuous and $\Hc^{2n+m}(B_\Nc((\zeta,x),R))=C R^{2n+2m}$ for every $(\zeta,x)\in \Nc$ and for every $R>0$, where $C=\Hc^{2n+m}(B_\Nc((0,0),1))$.  Since the reproducing kernel $k$ of $\Hc$ satisfies $k_{(\zeta,x)}=L_{(\zeta,x)} k_0$, also the first two assumptions of~\cite[Section 2.1 (B)]{Romeroetal} are clear. Finally, the third assmption of~\cite[Section 2.1 (B)]{Romeroetal} follows from the second one, Proposition~\ref{prop:11bis}, and Corollary~\ref{cor:1}.
\end{proof}

\end{document}